\documentclass[reqno]{amsproc}

\usepackage{graphicx}

\usepackage{hyperref}
\usepackage{graphicx}
\hypersetup{
colorlinks=true,
linkcolor=blue,
anchorcolor=blue,
citecolor=blue
}

\usepackage{amssymb}
\usepackage{amsfonts}
\usepackage{amsmath}

\newtheorem{theorem}{Theorem}[section]
\newtheorem{lemma}[theorem]{Lemma}
\newtheorem{proposition}[theorem]{Proposition}
\newtheorem{corollary}[theorem]{Corollary}

\theoremstyle{definition}
\newtheorem{definition}[theorem]{Definition}
\newtheorem{example}[theorem]{Example}

\newtheorem{remark}[theorem]{Remark}
\newtheorem*{theoremA}{Theorem A}

\numberwithin{equation}{section}

\begin{document}

\title{On exponential Yang-Mills fields and $p$-Yang-Mills fields}

\author{Shihshu Walter Wei}
\address{Department of Mathematics, University of Oklahoma, Norman, OK 73072}
\email{wwei@ou.edu}


\begin{abstract}
We introduce  \emph{normalized exponential Yang-Mills  energy
functional} $\mathcal{YM}_e^0$,  stress-energy tensor $S_{e,\mathcal{YM}^0 }$ associated with the normalized \emph{exponential Yang-Mills energy
functional} $\mathcal{YM}_e ^0 $,  $e$-conservation law.  
 We also introduce the notion of the {\it $e$-degree} $d_e$ which connects two separate parts in the associated normalize exponential stress-energy tensor $S_{e,\mathcal{YM}^0 }$  (cf. \eqref{3.10} and \eqref{4.15}), derive monotonicity formula for exponential Yang-Mills fields, and prove a vanishing theorem for exponential 
 Yang-Mills fields.
 These monotonicity formula and vanishing theorem for exponential 
 Yang-Mills fields augment and extend monotonicity formula and vanishing theorem for $F$-Yang-Mills fields in  \cite {DW} and \cite [9.2] {W11}. We also  
discuss an average principle (cf. Proposition \ref{P: 8.1}), isoperimetric and Sobolev inequalities, convexity and Jensen's inequality, $p$-Yang-Mills fields, an extrinsic average variational method  in the calculus of variation
 and $\Phi_{(3)}$-harmonic maps, from varied, coupled, generalized viewpoints and perspectives (cf. Theorems \ref{T: 6.1}, \ref{T: 7.1}, \ref{T: 9.1}, \ref{T: 9.2}, \ref{T: 10.1}, \ref{T: 10.2}, \ref{T: 11.13}, \ref{T: 11.14}, \ref{T: 11.15}). 
\end{abstract}

\maketitle

\section*{Contents}

1. Introduction

2. Fundamentals in vector bundles and principal $G$-bundle

3. Normalized exponential Yang-Mills functionals and $e$-conservation laws

4. Comparison theorems in Riemannian geometry

5. Monotonicity formulae

6. Vanishing theorems for exponential Yang-Mills fields

7. Vanishing theorems from exponential Yang-Mills fields to $F$-Yang-Mills fields

8.  An average principle, isoparametric and sobelov inequalities   

9.  Convexity and Jensen's inequalities

10. $p$-Yang Mills fields

11. An extrinsic average variation method and $\Phi_{(3)}$-harmonic maps

\section{Introduction}
The \emph{Yang--Mills functional}, brought to mathematics by physics is broadly analogous to functionals such as the \emph{length functional} in \emph{geodesic theory}, the \emph{area functional} in \emph{minimal surface, or minimal submanifold
    theory}, the \emph{energy $($resp. $p$-energy$)$ functional} \emph{in harmonic $($resp. $p$-harmonic$)$ map theory}, or the \emph{mass functional} in \emph{stationary or minimal current, geometric measure theory
}(cf.,e.g., \cite {FF, L, HoW}). A critical point of the Yang-Mills functional with respect to any compactly supported variation in the space of smooth connections  $\nabla$ on the adjoint bundle is called a\emph{Yang-Mills connection}.
Its associated curvature field $R^\nabla$ is known as\emph{Yang-Mills field} and is ``harmonic", i.e., a \emph{harmonic $2$-form with values in the vector bundle}. The Euler-Lagrange equation for the Yang-Mills functional is\emph{Yang-Mills equation}.
Whereas \emph{Hodge theory of harmonic forms} 
is motivated in part by\emph{Maxwell's equations of unifying}
magnetism with electricity in a physics world, and harmonic forms are privileged representatives in a \emph{de Rham
cohomology class} picked out by the Hodge Laplacian, harmonic maps can be viewed as a nonlinear generalization of harmonic $1$-form and Yang-Mills field can be viewed as a nonlinear generalization of harmonic $2$-form. On the other hand, Yang-Mills equation which can be viewed as a non-abelian generalization of Maxwell's equations, has had wide-ranging consequences, and influenced developments in other fields such as low-dimensional topology, particularly
the topology of smooth 4-manifolds.  For example, M. Freedman and R. Kirby first observed the startling fact that \emph{there exists an exotic $\mathbb R^4$}, i.e., a manifold homeomorphic to, but not diffeomorphic to, $\mathbb R^4$ (cf. \cite [p. 95] {K}, \cite {D, FK, Go}).
This is in stunning contrast to a phenomenal theorem of J. Milnor in compact high-dimensional topology which shows that \emph{there exist exotic seven-spheres $S^7$}, i.e., manifolds that are homeomorphic to, but not diffeomorphic to, the stadard Euclidean $S^7$
(cf. \cite {M}).

In \cite{DW}, we unify the concept of minimal hypersurfaces  in Euclidean space $\mathbb R^{n+1}$, maximal spacelike hypersurfaces in Minkowski space $\mathbb R^{n,1}$, harmonic maps, $p$-harmonic maps, $F$-harmonic maps, Yang-Mills fields, and introduce 
$F$-Yang-Mills fields, $F$-degree, and generalized Yang-Mills-Born-Infeld fields (with the plus sign or with the minus sign) on manifolds, where 
\begin{equation}
F: [0, \infty) \to [0, \infty)\, \text{is}\, \text{
a}\,  \text{strictly}\, \text{increasing}\, C^2\, \text{function}\,  
\text{with}\,  F(0)=0. \label{1.1}
\end{equation} When 
$$F(t)=t\, , p^{-1}(2t)^{\frac p2}\, , \sqrt{1+2t} -1\, , \text{and}\, 
1-\sqrt{1-2t},
$$  
$F$-Yang-Mills field becomes an ordinary Yang-Mills field, $p$-Yang-Mills field, a generalized Yang-Mills-Born-Infeld field with the plus sign, and a generalized Yang-Mills-Born-Infeld field with the minus sign on a manifold respectively (cf. \cite {BI,BL,BLS,CCW,D,La,LWW,LY, SiSiYa, W12,Ya}).
When $$F(t)=t\, , e^t, p^{-1}(2t)^{\frac p2}\, , \sqrt{1+2t} -1\, , \text{and}\ 1-\sqrt{1-2t}\, ,
$$  
$F$-harmonic map or the graph of $F$-harmonic map becomes an ordinary harmonic map, exponentially harmonic map, $p$-harmonic map, minimal hypersurface in Euclidean space $\mathbb R^{n+1}$, and maximal spacelike hypersurface in Minkowski space $\mathbb R^{n,1}$ respectively (cf. \cite {ES,WY,EL,Ar,WWZ}).
 
We use ideas from physics - \index{stress-energy tensors}{stress-energy tensors} and \index{conservation laws}{conservation laws} to simplify and unify various properties in $F$-Yang-Mills fields,  $F$-harmonic maps, and more generally differential $k$-forms, $k \ge 0$ with values in vector bundles.

In this paper, we introduce \emph{normalized exponential Yang-Mills  energy
functional} $\mathcal{YM}_e^0$, $\big ($resp. \emph{exponential Yang-Mills  energy
functional} $\mathcal{YM}_e\big )$, \emph{stress-energy tensor $S_{e,\mathcal{YM}^0 }$ associated with the normalized exponential Yang-Mills energy
functional} $\mathcal{YM}_e ^0 $, $\big ($resp. \emph{stress-energy tensor $S_{e,\mathcal{YM} }$ associated with the exponential Yang-Mills  energy functional} $\mathcal{YM}_e \big )$,  
( A critical point of $\mathcal{YM}_e^0$, i.e. a \emph{normalized exponential Yang-Mills connection}, and its associated \emph{normalized exponential Yang-Mills  field} are just the same as the \emph{Yang-Mills connection} and its associated \emph{exponential Yang-Mills field}).  We also introduce the notion of the {\it $e$-degree $d_e$} which connects two separate parts in the associated normalized exponential stress-energy tensors $S_{e,\mathcal{YM}^0 }$ $($cf. \eqref{4.15}$)$. 

These stress-energy tensors arise from calculating the rate of change of various functionals when the metric of the domain or base manifold is changed and are naturally linked to various conservation laws.
For example, we prove that 
every normalized exponential Yang-Mills
field or every exponential Yang-Mills
field $R ^\nabla $ satisfies an $e$-conservation law (cf. Theorem \ref{T: 3.11}). 
Every normalized exponential Yang-Mills connection or exponential Yang-Mills connection satisfies the exponential Yang-Mills equation (cf. Corollary \ref{C: 3.7}).
We then prove monotonicity
formulae, via the coarea formula and comparison theorems in Riemannian geometry (cf. \cite {GW, DW, HLRW, W11}).
Whereas a ``microscopic" approach to some of these monotonicity formulae leads to celebrated blow-up techniques due to E. de-Giorgi (\cite {Gi}) and W.L. Fleming (\cite {Fl}), and regularity theory in geometric measure theory(cf. \cite {A,Al,FF,HL,Lu,PS,SU}, for example, the regularity results of Allard (\cite {A}) depend on the monotonicity formulae for varifolds. Monotonicity properties are also dealt with by Price and Simon (\cite {PS}), Price (\cite {P}) for
Yang-Mills fields, and by Hardt-Lin (\cite {HL}) and Luckhaus (\cite {Lu}) for $p$-harmonic maps),
a ``macroscopic" version of these monotonicity formulae enable us to
derive some vanishing theorems under suitable growth conditions on
Cartan-Hadamard manifolds or manifolds which possess a pole with
appropriate curvature assumptions. In particular, we have Theorem \ref{T: 5.1} - {\it the monotonicity formula for exponential Yang-Mills fields} and Theorem \ref{T: 6.1} - {\it the vanishing theorem for exponential Yang-Mills fields.}

These monotonicity formula and vanishing theorem for exponential Yang-Mills fields augment and extend vanishing theorems for $F$-Yang-Mills fields in  \cite {DW} and \cite {W11}. We note that even when 
$$F(t)=e^t\quad \text{or}\quad F(t) = e^t -1\quad \text{for}\  t= \frac{||R^{\nabla}||^2}2,
$$
$F$-Yang-Mills field becomes exponential Yang-Mills field,
the following vanishing theorem for $F$-Yang-Mills fields are not applicable to exponential Yang-Mills fields, This is due to the fact that for $F(t) = e^t\, ,$
the degree of $F$, $$d_F := \sup_{t\geq 0}\frac{tF^{\prime }(t)}{F(t)} = \infty,$$ and the $F$-Yang-Mills energy functional growth condition \eqref{1.3} is not satisfied for $\lambda = -\infty$ in \eqref{1.4}.  To overcome this difficulty in getting estimates, we introduce the notion of e-degree $d_e\, ,$ for a given curvature tensor $R^\nabla \big ($cf. \eqref{4.15}$\big )$.

\begin{theoremA}[Vanishing theorem for $F$-Yang-Mills fields (\cite {DW, W11})]  Suppose that the radial curvature $K(r)$ of $M$ satisfies one of the seven conditions
\begin{equation}
\aligned
{\rm(i)}&\quad -\alpha ^2\leq K(r)\leq -\beta ^2\, \text{with}\, \alpha >0, \beta
>0\, \text{and}\,  (n-1)\beta -  4\alpha d_F  \geq 0;\\
{\rm(ii)}&\quad K(r) = 0\, \text{with}\,    n-4 d_F>0;\\
{\rm(iii)}&\quad -\frac A{(1+r^2)^{1+\epsilon}}\leq K(r)\leq \frac B{(1+r^2)^{1+\epsilon}} \text{with}\, \epsilon > 0\, , A \ge 0\, , 0 < B < 2\epsilon\, , \text{and}\\
&\qquad n - (n-1)\frac B{2\epsilon} -4e^{\frac {A}{2\epsilon}} d_F > 0;\\
{\rm(iv)}&\quad-\frac {A}{r^2}\leq K(r)\leq -\frac {A_1}{r^2}\,  \text{with}\quad  0 \le A_1 \le A\, ,\text{and}\\
&\qquad\   1 + (n-1)\frac{1 + \sqrt {1+4A_1}}{2} - 2(1 + \sqrt {1+4A}) d_F > 0;\\
{\rm(v)}&\quad- \frac {A(A-1)}{r^2}\le K(r) \le - \frac {A_1(A_1-1)}{r^2}\, \text{and}\,   A \ge A_1 \ge 1\, , \text{and}\\ 
&\qquad 1+(n-1)A_1-4A d_F > 0;\\
{\rm(vi)}&\quad \frac {B_1(1-B_1)}{r^2}\leq K(r)\le \frac {B(1-B)}{r^2}\, ,\text{with}\quad 0 \le B, \, B_1 \le 1\, ,  \text{and} \\
&\qquad 1 + (n-1)(|B-\frac {1}{2}|+ \frac {1}{2}) -2\big (1 + \sqrt {1+4B_1(1-B_1)}\big ) d_F > 0;\\
{\rm(vii)}&\quad \frac {B_1}{r^2} \le K(r) \le \frac {B}{r^2} \text{with}\quad  0 \le B_1 \le B \le \frac 14\, , \text{and}\\ 
 &\qquad 1+ (n-1)\frac{1 + \sqrt {1-4B}}{2} -(1 + \sqrt {1+4B_1} )\|R^\nabla \|^2 _\infty  > 0. 
\endaligned\label{1.2}
\end{equation} 
 
  If $R^\nabla \in A^2 \big (Ad(P)\big )$ is an $F$-Yang-Mills field  and satisfies  
\begin{equation}
\int_{B_\rho(x_0)}F (\frac{||R^\nabla||^2}2)\,  dv = o(\rho^\lambda )\quad \text{as
} \rho\rightarrow \infty,\label{1.3}
\end{equation}
where $\lambda $ is given by

\begin{equation}
\lambda \le \begin{cases}
n-4\frac \alpha \beta  d_F &\text {if } K(r)\ \text{obeys $($i$)$}\\
n-4d_F &\text {if } K(r)\ \text{obeys $($ii$)$}\\
n - (n-1)\frac B{2\epsilon} -4e^{\frac {A}{2\epsilon}}d_F &\text{if } K(r)\ \text{obeys $($iii$)$}\\
1 + (n-1)\frac{1 + \sqrt {1+4A_1}}{2} - 2 (1 + \sqrt {1+4A}) d_F  &\text{if } K(r)\, \text{
obeys $($iv$)$}\\
1+(n-1)A_1-4A d_F &\text{if } K(r)\, \text{
obeys $($v$)$}\\
1 + \frac{n-1}{(|B-\frac {1}{2}|+ {2}^{-1})^{-1}} - 2 \big (1 + \sqrt {1+4B_1(1-B_1)}\big ) d_F &\text{if } K(r)\ \text
{obeys $($vi$)$}\\
1+ (n-1)\frac{1 + \sqrt {1-4B}}{2} - 2 (1 + \sqrt {1+4B_1} ) d_F &\text{if } K(r)\ \text{obeys $($vii$)$}.\end{cases}
\label{1.4}
\end{equation}
 Then $R^{\nabla} \equiv 0$ on $M\, .$
 In particular, every $F$-Yang-Mills field $R^{\nabla}$  with finite $F$-Yang-Mills energy functional vanishes on $M$.
\end{theoremA}
We also discuss An Average Principle (cf. Proposition {8.1})
and Jensen's inequality from varied, generalized viewpoints and perspectives of exponential Yang-Mills fields, $p$-Yang-Mills fields, and Yang-Mills fields.
(Theorems \ref{T: 7.1}, \ref {T: 9.1}, \ref {T: 9.2}, \ref {T: 10.1}, \ref {T: 10.2}). 

In the context of harmonic
maps, the stress-energy tensor was introduced and studied in
detail by Baird and Eells (\cite {BE}). Following Baird-Eells (\cite {BE},
Sealey [Se2] introduced the stress-energy tensor for vector
bundle valued $p$-forms and established some vanishing theorems
for $L^2$ harmonic $p$-forms (cf. \cite {DLW, Se1, Xi1}).  In a more general frame, Dong and Wei use a unified method to study the stress-energy tensors
and yields monotonicity
inequalities, and vanishing theorems for vector bundle valued $p$-forms (\cite {DW}). The idea and methods can be extended and unified in $\sigma_2$-version of harmonic maps - 
$\Phi$-Harmonic maps (cf. \cite {HW}). These are the second elementary symmetric function of a pull-back tensor, whereas harmonic maps are the first elementary symmetric function of a pull-back tensor.
More recently, Feng-Han-Li-Wei use stress-energy tensors to unify properties in $\Phi_S$-harmonic maps (cf. \cite {FHLW}), Feng-Han-Wei extend and unify results in $\Phi_{S,p}$-harmonic 
maps (cf. \cite {FHW}), and Feng-Han-Jiang-Wei further extend and unify results in $\Phi_{(3)}$-harmonic maps (cf. \cite{FHJW}). Whereas we can view harmonic maps as $\Phi_{(1)}$-harmonic maps $($involving $\sigma_1$$)$ and $\Phi$-harmonic maps as $\Phi_{(2)}$-harmonic maps $($involving $\sigma_2$$)$, $\Phi_{(3)}$-harmonic maps involve $\sigma_3$, the third elementary symmetric function of the pullback tensor. In fact, an extrinsic average variational method in the calculus of variation can be carried over to more general settings by which we introduce a notion of $\Phi_{(3)}$-harmonic map and find a large class of manifolds, $\Phi_{(3)}$-superstrongly unstable $(\Phi_{(3)}$-$\text{SSU})$ manifolds, introduce notions of a stable $\Phi_{(3)}$-harmonic map, and $\Phi_{(3)}$-strongly unstable $(\Phi_{(3)}$-$\text{SU})$ manifolds (cf. Theorems \ref{T: 11.8}, \ref{T: 11.9}, \ref{T: 11.10}, 
and \ref{T: 11.11}).

\noindent
By an extrinsic average variational method in the calculus variations proposed in \cite {W3}, we find 
multiple large classes of manifolds with geometric and topological properties in the setting of varied, coupled, generalized type of harmonic maps, and summarize some of the results in Table 1. For some details, related ideas, techniques, we refer to 
\cite {CW3}, \cite {W1}-\cite {W12}, \cite {WLW}.

{\fontsize{0.01}{8}\selectfont
\begin{table}[ht]
\caption{An Extrinsic Average Variational Method}\label{eqtable}
\renewcommand\arraystretch{1.2}
\noindent\[
\begin{array}{|c|c|c|c|c|}
\hline
\operatorname{Mappings}&\operatorname{Functionals}&\operatorname{New}\, \operatorname{manifolds}\, \operatorname{found}&\operatorname{Geometry}&\operatorname{Topology}\\
\hline
\operatorname{harmonic}\,\operatorname{map}\, \operatorname{or}&\operatorname{energy}\, \operatorname{functional}\,E\, \operatorname{or}&\operatorname{SSU}\, \operatorname{manifolds}\, \operatorname{or}&\operatorname{SU}\, \operatorname{or}&
 \pi_1=\pi_2=0\\
\Phi_{(1)}-\operatorname{harmonic}\,\operatorname{map} & E_{\Phi_{(1)}}&\Phi_{(1)}-\operatorname{SSU}\, \operatorname{manifolds}&\Phi_{(1)}-\operatorname{SU}& 
\pi_1=\pi_2=0\\ \hline
p-\operatorname{harmonic}\,\operatorname{map}&E_p&p-\operatorname{SSU}\, \operatorname{manifolds}&p-\operatorname{SU}& \pi_1=\cdots=\pi_{[p]}=0\\
\hline
\Phi-\operatorname{harmonic}\,\operatorname{map}\, \operatorname{or}&\Phi-\operatorname{energy}\, \operatorname{functional}\,E_{\Phi}\, \operatorname{or}&\Phi-\operatorname{SSU}\, \operatorname{manifolds}\, \operatorname{or}&\Phi-\operatorname{SU}\, \operatorname{or}& \pi_1=\cdots=\pi_4=0\\
\Phi_{(2)}-\operatorname{harmonic}\,\operatorname{map}&E_{\Phi_{(2)}}&\Phi_{(2)}-\operatorname{SSU}\, \operatorname{manifolds}&\Phi_{(2)}-\operatorname{SU}& \pi_1=\cdots=\pi_4=0\\
\hline
\Phi_S-\operatorname{harmonic}\,\operatorname{map}&E_{\Phi_S}&\Phi_S-\operatorname{SSU}\, \operatorname{manifolds}&\Phi_S-\operatorname{SU}& \pi_1=\cdots=\pi_4=0\\
\hline
\Phi_{S,p}-\operatorname{harmonic}\,\operatorname{map}&E_{\Phi_{S,p}}&\Phi_{S,p}-\operatorname{SSU}\, \operatorname{manifolds}&\Phi_{S,p}-\operatorname{SU}& \pi_1=\cdots=\pi_{[2p]}=0\\
\hline
\Phi_{(3)}-\operatorname{harmonic}\,\operatorname{map}&\Phi_{(3)}-\operatorname{energy}\, \operatorname{functional}\,E_{\Phi_{(3)}}&\Phi_{(3)}-\operatorname{SSU}\, \operatorname{manifolds}&\Phi_{(3)}-\operatorname{SU}& \pi_1=\cdots=\pi_{6}=0\\
\hline
\end{array}
\]
\end{table}
}

\section{Fundamentals in vector bundles and principal $G$-bundle}

This section is devoted to a brief discussion of the fundamental notions in vector bundles and principal $G$-bundle.

\begin{definition}

A $($differentiable$)$ {\it vector bundle} of rank $n$ consists of a total space $E$, a base $M$, and a projection $\pi : E \to M\, ,$ where $E$ and $M$ are differentiable manifolds, $\pi$ is differentiable, each fiber $E_x := \pi^{-1} (x)$ for $x \in M$, carries the structure of an $n$-dimensional (real) vector space, with the following local triviality: For each $x\in M$, there exist a neighborhood $U$ and a diffeomorphism $$\varphi:   \pi^{-1} (U) \to U \times \mathbb R^n$$ such that  for every $y\in U$ $$\varphi _y  
:= \varphi_{|E_y} : E_y \to \{y\} \times \mathbb R^n$$ is a vector space isomorphism. Such a pair $(\varphi, U)$ is called a {\it bundle chart}.
\end{definition}

Note that local trivializations $\varphi_ {\alpha}, \varphi_ {\beta}$ with $U_ {\alpha} \cap U_{\beta} \ne \emptyset$ determines {\it transition maps}
\[\varphi_ {\beta \alpha } :  U_ {\alpha} \cap U_{\beta}  \to \text {Gl}(n, \mathbb R)
\]
by 
\begin{equation*}\varphi_ {\beta} \circ \varphi_ {\alpha}^{-1} (x, v)   =(x, \varphi_ {\beta \alpha} (x) v)\quad \text{for}\quad x\in M, v \in \mathbb R^n\, , 
\end{equation*}
where $\text {Gl}(n, \mathbb R)$ is the general linear group of bijective linear self maps of $\mathbb R^n\, .$

As direct consequences, the transition maps satisfy:
\begin{equation*}
\begin{aligned}
\varphi_ {\alpha\alpha} (x) &= \text{id}_{\mathbb R^n}\quad \text{for}\quad x\in U_{\alpha};\\ 
 \varphi_ {\alpha \beta} (x)\varphi_ {\beta \alpha}(x) &= \text{id}_{\mathbb R^n}\quad \text{for}\quad x\in U_ {\alpha} \cap U_{\beta};\\
 \varphi_ {\alpha \gamma}(x)\varphi_ {\gamma \beta} (x)\varphi_ {\beta \alpha} (x) &= \text{id}_{\mathbb R^n}\quad \text{for}\quad x\in U_ {\alpha} \cap U_{\beta}\cap U_{\gamma}.
\end{aligned}
\end{equation*}

(cf. \cite {J}) A vector bundle can be reconstructed from its transition maps

\begin{equation*} E = \coprod _{\alpha} \quad U_{\alpha} \times \mathbb R^n\, / \, \sim\, ,
\end{equation*}
where $\coprod$ denotes disjoint union, and the equivalence relation $\sim$ is defined by
\begin{equation}
 (x, v) \sim (y, w) \  : \Longleftrightarrow\  x= y \ \text{and} \ w = \varphi _{\beta\alpha} (x) v\  (x \in U_{\alpha}, y \in U_{\beta}, v, w \in \mathbb R^n)\, .\label{2.1}
\end{equation}
 
\begin{definition}
Let $G$ be a subgroup of $\text {Gl}(n, \mathbb R)$, for example the orthogonal group $O(n)$ or special orthogonal group $SO(n)\, .$ By a vector bundle  has {\it the structure group $G$}, we mean there exists an atlas of bundle charts for which all 
transition maps have their values in $G\, .$
\end{definition}
 
\begin{definition}
Let $G$ be a Lie group.  A  {\it principal $G$-bundle} consists of a base $M$, the total space $P$ of the bundle, and a differentiable projection $\pi : P \to M\, ,$ where $P$ and $M$ are differentiable manifolds, with an action of $G$ on $P$ satisfying 
\smallskip

\item(i) $G$ acts freely on $P$ from the right: $(q,p)\in P \times G$ is mapped to $qp\in P\, ,$ and $qp \ne q$ for $q \ne e\, .$\smallskip

\noindent
The $G$ action then defines an equivalence relation on $P:  p \sim q : \Longleftrightarrow \exists g\in G$ such that $p=qg\, .$\smallskip

\item(ii) $M$ is the quotient of $P$ by this equivalence relation, and $\pi : P \to M$ maps $q\in M$ to its equivalence class. By $(\operatorname{i})$, each fiber $\pi^{-1} (x)$ can then be identified with $G$.
\smallskip

\item(iii)
$P$ is locally trivial in the following sense:

\noindent
For each $x\in M$, there exists a neighborhood $U$ of $x$ and a diffeomorphism 
\[\varphi:   \pi^{-1} (U) \to U \times G
\]
of the form $\varphi(p)=(\pi(p), \psi(g))$ which is $G$-equivariant, i.e. $\varphi (pg) = (\pi(p), \psi(p)g)$ for all $g \in G.$
\end{definition}

\begin{example} We have the following results.
\begin{enumerate}
\item[\rm(i)] The projection $S^n \to P^n(\mathbb R)$ of the $n$-sphere to the real projective space is a principal bundle with group $G = O(1) = Z_2$
\smallskip

\item[\rm(ii)] The Hopf map $S^{2n+1} \to P^n(\mathbb C)$ of the $2n+1$-sphere to the complex projective space is a principal bundle with group $G = U(1) = S^1$
\smallskip

\item[\rm(iii)] The Hopf map $S^{4n+1} \to P^n(\mathbb Q)$ of the $4n+1$-sphere to the quaternionic projective space is a principal bundle with group $G = Sp(1) = S^3$
\smallskip

\item[\rm(iv)] Hopf fibrations: $ S^1 \to S^1, S^3 \to S^2, S^7 \to S^4,$ and $ S^{15}
\to S^8\, $
\end{enumerate}
For $k=1, 2, 4, 8\, ,$   the
Hopf construction is defined
 by $$(z,w)\mapsto u(z, w) = (|z|^2 - |w|^2,
2 z \cdot \overline w): \ \ \mathbb{R}^k \times \mathbb{R}^k \to \mathbb{R}^{k+1}.
$$

In fact, Hopf fibrations are $p$-harmonic maps and  $p$-harmonic morphisms for every $p > 1$ (c.f., e.g., \cite {W8, CW2}).

\noindent
We recall a $C^2$ map $u: M \to N$ is said to be a \emph{p-harmonic
morphism} if for any p-harmonic function $f$ defined on an open
set $V$ of $N$, the composition $f \circ u$ is $p$-harmonic on
$u^{-1}(V)$.
\label{E: 2.4}
\end{example}

\begin{example} If $E \to M$ is a vector bundle with fiber $V$, the bundle of bases of $E, B(E) \to M$ is a principle bundle with group $\operatorname{Gl}(V)\, .$  
\end{example}

\subsection{Reversibility of principal and vector bundles}
\medskip

\noindent
$(\Longrightarrow)$ Given a principal $G$-bundle $P \to M$ and a vector space $V$ on which $G$ acts from the left, we construct the associated vector bundle $E \to M$ with fiber $V$ as follows:

\noindent
We have a free action of $G$ on $P \times V$ from the right:

\[
\begin{aligned}
P \times V \times G &\to P \times V\\
(p,v) \cdot g & = (p \cdot g, g^{-1} v)\, .
 \end{aligned}
\]
If we divide out this $G$-action, i.e. identify  $(p,v)$ and $(p,v) \cdot g $, the fibers of $(P \times V) / G \to P/G$ becomes vector spaces isomorphic to $V$, and 

\begin{equation*} E := P \times _{G} \, V\, := (P \times V)_{/G} \to  M 
\end{equation*}
is a vector bundle with fiber $G \times _G \, V\,  := (G \times  V)_{/G} = V$ and structure group $G\, .$ The transition functions for $P$ also give transition functions for $E$ via  the left action of $G$ on $V.$

\noindent
$(\Longleftarrow)$ 
Conversely, given a vector bundle $E$ with structure group $G$, we construct a principal $G$-bundle as
\begin{equation*}\coprod _{\alpha} \quad U_{\alpha} \times G\, / \, \sim
\end{equation*}
with
\[  (x_\alpha, g_\alpha) \sim (x_\beta, g_\beta) \quad  : \Longleftrightarrow\quad  x_\alpha = x_\beta \in U_{\alpha} \cap U_{\beta}\quad \text{and} \quad g_\beta = \varphi _{\beta\alpha} (x) g_\alpha\]  where $\{U_{\alpha}\}$ is a local trivialization of $E$
with transition functions $\varphi_{\beta \alpha}$ as in \eqref{2.1}.
\smallskip

\begin{example} We have the following assertions.

\begin{itemize}
\item[\rm(i)] The canonical line bundles (real, complex and  quaternionic) over the projective spaces $P^n(\mathbb R)\, ,P^n(\mathbb C)\, $ and of the $P^n(\mathbb Q)$ are the {\it associated bundles} of 
the principal bundles in Example \ref{E: 2.4} $(\operatorname{i})-(\operatorname{iii})$ via the canonical actions of $O(1), U(1)$ and $Sp(1)$ on $\mathbb R, \mathbb C$ and $\mathbb Q$ respectively.

\item[\rm(ii)] Let $E \to M$ be a bundle with fiber $F$ and structure group $G$ and $f: N \to M$ be a map between manifolds $N$ and $M$. Then the pull-back of $E \to M$ is a bundle $f^{-1} E \to M$ with fiber $F$, structure group $G$, and 
bundle charts $(\varphi \circ f, f^{-1}(U))$, where $\varphi (U)$ are bundle charts of $E\, .$ The pull-back $f^{-1} E \to M$ is called the {\it pull-back bundle}.
\end{itemize}
\end{example}

\section{ Normalized exponential Yang-Mills functionals and $e$-conservation laws}

Our basic set-up is the following: We consider a Riemannian manifold $M$, and a principal bundle $P$
with compact structure Lie group $G$ over $M$.
Let $Ad(P)$ be the adjoint bundle
\begin{equation}
Ad(P)=P\times _{Ad}\mathcal{G}\, , \label{3.1}
\end{equation}
where $\mathcal{G} $ \ is the Lie algebra of $G$. Every connection 
$\rho$
on $P$ induces a connection $\nabla$ on $Ad(P)$. A connection $\nabla$ on the vector bundle $Ad(P)$ is a rule that equips us to take derivatives of smooth cross sections of $Ad(P)$. We also have the
Riemannian connection $\nabla ^M$ on the tangent bundle $TM$, and
the induced connection on the tensor product $\Lambda ^2T^{*}M\otimes Ad(P)$, where  $\Lambda ^2T^{*}M$ is the second exterior power of the cotangent bundle $T^{*}M$. An
$Ad_G$ invariant inner product on $\mathcal{G}$ induces a fiber metric on
$Ad(P)$ and makes $Ad(P)$ and $\Lambda ^2T^{*}M\otimes Ad(P)$
into Riemannian vector bundles. Denote by $\Gamma \big (\Lambda ^2T^{*}M\otimes Ad(P)\big )$ the (infinite-dimensional) vector space of smooth sections of $\Lambda ^2T^{*}M\otimes Ad(P)\, .$ 
 For $k\ge 0$ set $$A^k\big (Ad(P)\big)=\Gamma (\Lambda
^kT^{*}M\otimes Ad(P))$$ be the space of smooth $k$-forms on $M$ with
values in the vector bundle $Ad(P)$. 
Although $\rho$ is not a section
of $A^1\big (Ad(P)\big)$ , via its induced connection $\nabla $, the associated curvature tensor $ R^\nabla$, given by $$R^{\nabla}_{X,Y} = [\nabla_X, \nabla_Y]- \nabla_{[X,Y]}\, ,$$ is
in $A^2(Ad(P))$. Let $\mathcal{C}$ be the space of smooth connections $\nabla \, $ on $Ad(P)\, ,$ and $dv$ be the volume element of $M\, .$ Recall
the \emph{Yang-Mills functional} is the mapping $\mathcal{YM} : \mathcal{C}\to \mathbb{R}^+\, $ given by
\begin{equation}
\mathcal{YM}(\nabla )=\int_M \frac 12||R ^\nabla ||^2\, dv\, ,\label{3.2}
\end{equation}
the \emph{$p$-Yang-Mills functional}, for $p \ge 2\, $  (resp. the \emph{$F$-Yang-Mills functional}) is the mapping $\mathcal {YM}_p : \mathcal{C}\to \mathbb{R}^+\, $ given by
\begin{equation}\aligned
\mathcal{YM}_p(\nabla )& =\int_M \frac 1p||R ^\nabla ||^p\, dv\,  \\
\big (resp.\quad \mathcal{YM}_F(\nabla )& =\int_M F(\frac 12||R ^\nabla ||^2)\, dv\, \big ),
\endaligned
\label{3.3}
\end{equation}
where the norm is defined in terms of the Riemannian metric on $M$ and a fixed $Ad_G$-invariant inner product on the Lie algebra $\mathcal{G}$ of $G\, .$ That is, at each point $x\in M\, ,$ its norm
\begin{equation}||R^{\nabla}||^2_x = \sum_{i<j}||R^{\nabla}_{e_i,e_j}||^2_x\, \label{3.4}
\end{equation}
where $\{e_1, \cdots, e_n\}$ is an orthonormal basis of $T_x(M)$ and the norm of $R^{\nabla}_{e_i,e_j}$ is the standard one on Hom$(Ad(P), Ad(P))$-namely, $$\langle S, U\rangle \equiv \, \text {trace}\, (S^T \circ U)\, .$$

A connection $\nabla$ on the adjoint bundle $Ad(P)$ is said to be a \emph{Yang-Mills connection} (resp. \emph{$p$-Yang-Mills connection}, $p \ge 2$, \emph{$F$-Yang-Mills connection}) and its associated curvature tensor $R^{\nabla}$ is said to be a \emph{Yang-Mills field} (resp. \emph{$p$-Yang-Mills field}, $p \ge 2$, \emph{$F$-Yang-Mills field}), if $\nabla$ is a critical point of $\mathcal{YM}$ (resp. $\mathcal{YM}_p\, $, ${\mathcal{YM}}_F$) with respect to any compactly supported variation in the space of smooth connections on $Ad(P)$ . We now introduce

\begin{definition} 
The \emph{normalized exponential Yang-Mills  energy
functional} is the mapping $\mathcal{YM}_e ^0 : \mathcal{C}\to \mathbb{R}^+\, $ given by
\begin{equation}
\mathcal{YM}_e ^0(\nabla )=\int_M \big (\exp (\frac 12||R ^\nabla ||^2) - 1\big )\, dv\, , \label{3.5}
\end{equation}
the \emph{exponential Yang-Mills energy
functional} is the mapping $\mathcal{YM}_e : \mathcal{C}\to \mathbb{R}^+\, $ given by
\begin{equation}
\mathcal{YM}_e(\nabla )=\int_M \exp (\frac 12||R ^\nabla ||^2)\, dv\, , \label{3.6}
\end{equation}
\end{definition}

\noindent
on $M\, ,$ 
the uniform norm $||R^{\nabla}||_\infty$ is given by

\begin{equation}||R^{\nabla}||^2_\infty = \sup_{x \in M}||R^{\nabla}||^2_x \, .\label{3.7}
\end{equation}

The normalized exponential Yang-Mills  energy
functional $\mathcal{YM}_e ^0$ has the following simple and useful advantage.

\begin{proposition}  
\begin{equation}\mathcal{YM}_e ^0(\nabla ) \ge 0\quad \text{and}\quad \mathcal{YM}_e ^0(\nabla )=0\quad \Longleftrightarrow\quad R ^\nabla \equiv 0\, .\label{3.8}\end{equation}
This is an analog of $p$-Yang-Mills  functional, for $p \ge 2\, ,$ 
\begin{equation}\mathcal{YM}_p(\nabla ) \ge 0\quad \text{and}\quad \mathcal{YM}_p(\nabla )=0\quad \Longleftrightarrow\quad R ^\nabla \equiv 0\, .\label{3.9}\end{equation}
\end{proposition}

\begin{definition} 
The {\it stress-energy tensor} $S_{e,\mathcal{YM}^0 }$ associated with the normalized {\it exponential Yang-Mills energy
functional} $\mathcal{YM}_e ^0 $ and the  {\it stress-energy tensor} $S_{e,\mathcal{YM} }$ associated with the {\it exponential Yang-Mills energy
functional} $\mathcal{YM}_e $ are
defined respectively as follows:
\begin{equation}\begin{aligned}
S_{e,\mathcal{YM}^0 }(X,Y)&=\big (\exp (\frac{||R^{\nabla}||^2}2) -1 \big ) g(X,Y)-\exp (\frac{||R^{\nabla}||^2}2)\langle i_XR^{\nabla} ,i_YR^{\nabla} \rangle \, ,  \\
\end{aligned}\label{3.10}
\end{equation}
\begin{equation}\begin{aligned}
S_{e,\mathcal{YM} }(X,Y)&=\exp (\frac{||R^{\nabla}||^2}2)\big (g(X,Y)- \langle i_XR^{\nabla} ,i_YR^{\nabla} \rangle \big )   
\end{aligned}\label{3.11}
\end{equation}
where $\langle\quad , \quad \rangle$ is the induced inner product on $A^{1}\big (Ad(P)\big )\, ,$ and $i_XR^{\nabla}$ is the interior multiplication by the vector
field $X$ given by
\begin{equation}
(i_XR^{\nabla})(Y_1)=R^{\nabla}(X,Y_1)\, ,\label{3.12}
\end{equation}
for any vector fields $Y_{1}$ on $M$.
\end{definition}
\medskip

{We calculate the rate of change of the
\emph{normalized exponential Yang-Mills energy
functional} $\mathcal{YM}^0_{e,g} $  and \emph{exponential Yang-Mills energy
functional} $\mathcal{YM}_{e,g} $ when the metric $g$ on the domain or base manifold is changed. To this end, we consider a compactly supported smooth one-parameter variation of the metric $g\, ,$ i.e. a smooth family of metrics $g_s$ such that $g_0=g\, .$  Set $\delta g =\frac {\partial g_s}{
\partial s}_{\big {|}_{s =0}}\, .$ Then $\delta g$ is a smooth $2$-covariant symmetric tensor field on
$M$ with compact support. These give birth to their associated stress-energy tensors.}

\begin{lemma}   
With the same notations as above, we have
\begin{equation}\begin{aligned}
\frac{d}{ds}
\mathcal{YM}_{e,g_s }^0({\nabla} )_{\big {|}_{s =0}}  & = \frac 12\int_M\langle S_{e,\mathcal{YM}^0 },\delta g
\rangle dv_g\\
\end{aligned}\label{3.13}\end{equation}
\begin{equation}\begin{aligned}
\frac{d}{ds}
\mathcal{YM}_{e,g_s }({\nabla} )_{\big {|}_{s =0}} & = \frac 12\int_M\langle S_{e,\mathcal{YM} },\delta g
\rangle dv_g
\end{aligned}\label{3.14}\end{equation}
where $S_{e,\mathcal{YM}^0 }$ and $S_{e,\mathcal{YM}}$ are as in \eqref{3.8} and \eqref{3.9} respectively.
\end{lemma}
\begin{proof} From (\cite {Ba}), we obtain
\begin{equation}\frac{d||R^{\nabla} ||_{g_s }^2}{ds
}_{\big {|}_{s =0}}=-\sum _{i,j} \langle i_{e_i}R^{\nabla} ,  i_{e_j}R^{\nabla}\rangle  \delta g(e_i,e_j)\label{3.15}
\end{equation}
and
\begin{equation}\frac d{ds } \, dv_{g_s}\,
_{\big {|}_{s =0}}=\frac 12\langle g,\delta g \rangle dv_g\, .\label{3.16}
\end{equation}
Then by the chain rule, \eqref{3.15}, \eqref{3.16}, and \eqref{3.10},  we have
\begin{equation}\aligned \frac{d}{ds}
\mathcal{YM}_{e,g_s }^0(\nabla) _{\big {|}_{s =0}}  
& = \int_M \frac {d}{ds
} \bigg( \exp (\frac{||R^{\nabla} ||_{g_s }^2}{2}\big) - 1 \, dv_{g_s}\bigg ) _{\big {|}_{s =0}}\\
& = \int_M\exp (\frac{||R^{\nabla} ||^2}2)\frac d{ds
} \big(\frac{||R^{\nabla} ||_{g_s }^2}2\big) {\big |}_{s =0}\, dv_g\\
& \qquad +\int_M \big (\exp (\frac 12||R ^\nabla ||^2) - 1\big )\frac d{ds } \, dv_{g_s}\, _{\big {|}_{s =0}} \\
&=\frac 12\int_M \bigg (\big (\exp (\frac 12||R ^\nabla ||^2) - 1\big )\langle g,\delta g \rangle \\
& \qquad -\exp (\frac{||R^{\nabla} ||^2}2) \sum _{i,j} \langle i_{e_i}R^{\nabla} ,  i_{e_j}R^{\nabla}\rangle  \delta g(e_i,e_j)\bigg )\, dv_g \\
&=\frac 12\int_M\langle S_{e,\mathcal{YM}^0  },\delta g \rangle dv_g\, .
\endaligned\label{3.17}\end{equation}
Similarly, we can calculate $\frac{d}{ds} 
\mathcal{YM}_{e,g_s }(\nabla)_{\big {|}_{s =0}}$ and obtain the desired \eqref{3.14}.
\end{proof}

The {\it exterior
differential operator} $d^\nabla :A^1\big (Ad(P)\big )\rightarrow
A^2\big (Ad(P)\big )$ relative to the connection $\nabla$ is given by
\begin{equation}
(d^\nabla \sigma
)(X_1, X_2)=(\nabla _{X_1}\sigma
)(X_2) - (\nabla _{X_2}\sigma
)(X_1) \, .  \label{3.18}
\end{equation}

Relative to the Riemannian structures of $Ad(P)$ and $TM$, the
{\it codifferential operator} $\delta ^\nabla : A^2\big (Ad(P)\big )\rightarrow
A^1\big (Ad(P)\big )$ is characterized as the adjoint of $d$ via the
formula
\begin{equation}
\int_M\langle d^\nabla \sigma ,\rho \rangle dv_g=\int_M\langle
\sigma ,\delta ^\nabla \rho\rangle dv_g\, , \label{3.19}
\end{equation}
where $\sigma \in A^1\big (Ad(P)\big ),\rho \in A^2\big (Ad(P)\big )$ , one of which
has compact support and $dv_g$ is the volume element associated with the metric $g$ on $TM$. Then
\begin{equation}
(\delta ^\nabla \rho )(X_1)=-\sum_i(\nabla
_{e_i}\rho )(e_i,X_1)\, .\label{3.20}  
\end{equation}

\begin{definition}  
A connection $\nabla$ on the adjoint bundle $Ad(P)$ is said to be an {\it exponential Yang-Mills connection} and its associated curvature tensor $R^{\nabla}$ is said to be an \emph{exponential Yang-Mills field}, if $\nabla$ is a critical point of $\mathcal{YM}_e$ with respect to any compactly supported variation in the space of connections on $Ad(P)$ . 
\end{definition}  \medskip

\begin{lemma}[The first variation formula for
normalized exponential Yang-Mills functional $\mathcal{YM}_e ^0$ or $\mathcal{YM}_e$]
Let $A \in A^1\big (Ad(P)\big )$ and $\nabla ^t
=\nabla +t A $ be a family of connections on $Ad(P)$. Then
\begin{equation}\aligned \frac d{dt}\mathcal{YM}_e ^0(\nabla ^t)_{\big {|}_{s =0}} = \frac d{dt}\mathcal{YM}_e ^0(\nabla ^t)_{\big {|}_{s =0}} =\int_M\langle\delta ^\nabla
\big(\exp (\frac 12||R
^\nabla ||^2)R ^\nabla \big), A \rangle \, dv\, .
\endaligned\label{3.21}\end{equation}

Furthermore, The Euler-Lagrangian equation for $\mathcal{YM}_e ^0$ or $\mathcal{YM}_e$ is
\begin{equation}
\exp (\frac 12||R ^\nabla ||^2)\delta ^\nabla R
^\nabla -i_{\text{grad} \big( \exp (\frac 12||R^\nabla ||^2)\big)}R
^\nabla =0\, ,  \label{3.22}
\end{equation}
or \begin{equation}
\delta ^\nabla  \big(\exp (\frac 12||R ^\nabla
||^2)R^\nabla  \big) = 0\, .\label{3.23}
\end{equation}\label{L: 3.6}
\end{lemma}

\begin{proof} By assumption, the curvature of $\nabla^t$ is given by  
\begin{equation}
 R ^{\nabla^t} =  R ^\nabla + t (d^\nabla A) + t^2 [A, A]\, ,
\label{3.24}
\end{equation}
where $[A, A]\in A^2(Ad(P))$ is given by $[A, A]_{X,Y}=[A_X, A_Y]\, .$ Indeed, for any local vector fields $X,Y$ on $M$. with $[X,Y]=0\, ,$ we have via \eqref{3.18}
\begin{equation}
\begin{aligned}
 R ^{\nabla^t}_{X,Y} &=(\nabla _X +t A_X)(\nabla_Y + t A_Y)-(\nabla _Y + t A_Y)(\nabla _X + t A_X)\\
&= R ^\nabla_{X,Y} + t [\nabla _X , A_Y ] - t [\nabla _Y , A_X ] + t^2 [A_X, A_Y]\\
&=  R ^\nabla_{X,Y} + t \nabla _X (A_Y) - t \nabla _Y (A_X) + t^2 [A, A]_{X,Y}\\
&=  R ^\nabla_{X,Y} + t (d^\nabla A)_{X,Y} + t^2 [A, A]_{X,Y}\, .\\
\end{aligned}
\label{3.25}
\end{equation}
Thus,
\begin{equation}
\aligned \exp\, (\frac 12|| R^{\nabla^t} ||^2)=\exp\, (\frac 12||R ^\nabla
||^2 +t \langle R ^\nabla ,d^\nabla A \rangle + \varepsilon(t^2))\, ,
\endaligned
\label{3.26}
\end{equation}
where $\varepsilon (t^2) = o (t^2)\quad \text{as}\, t \to 0\, . $
Therefore,
\begin{equation}
\mathcal{YM}_e(\nabla ^t)=\int _M \exp\, (\frac 12||R ^\nabla
||^2+t\langle R ^\nabla ,d^\nabla A \rangle+\varepsilon(t^2) )\,  dv
\label{3.27}
\end{equation}
and via \eqref{3.19}, we have
\begin{equation}
\aligned \frac d{dt}\mathcal{YM}_e ^0(\nabla ^t)_{\big {|}_{s =0}} & = \frac d{dt}\mathcal{YM}_e (\nabla ^t)_{\big {|}_{s =0}}\\
& =\int_M\ exp\, (\frac
12||R ^\nabla ||^2)\langle R^\nabla ,d^\nabla A \rangle \, dv\\
&=\int_M\langle\delta ^\nabla \big(\exp\, (\frac 12||R
^\nabla ||^2)R ^\nabla \big), A \rangle \, dv\, .
\endaligned
\label{3.28}
\end{equation}
This derives the Euler-Lagrange equation for $\mathcal{YM}_e ^0$ or $\mathcal{YM}_e$ by \eqref{3.20} as follows
\begin{equation}
\aligned 0 &=\delta ^\nabla \big (\exp\, (\frac 12||R ^\nabla
||^2)R^\nabla \big) \\
&= - \sum_{i=1}^m \big (\nabla _{e_i}\exp\, (\frac 12||R ^\nabla ||^2)R
^\nabla\big )(e_i,\cdot ) \\
&=\exp\, (\frac 12||R ^\nabla ||^2)\delta ^\nabla R
^\nabla -i_{\text{grad}\big (\exp\, (\frac 12||R^\nabla ||^2)\big )}R
^\nabla\, .
\endaligned
\label{3.29}
\end{equation}
\end{proof}

\begin{corollary}
Every normalized exponential Yang-Mills
connection or every exponential Yang-Mills
connecton $\nabla $ satisfies \eqref{3.29}. 
\label{C: 3.7} 
\end{corollary}

Dong and Wei derive
\smallskip

\noindent {\bf Theorem B (\cite {DW})}  {\it 
${\rm(i)}$ The Euler-Lagrangian equation for $F$-Yang-Mills functional $\mathcal{YM}_F$ is
\begin{equation}
F^{\prime }(\frac 12||R ^\nabla ||^2)\delta ^\nabla R
^\nabla -i_{\text{grad} \big(F^{\prime }(\frac 12||R^\nabla ||^2)\big)}R
^\nabla =0  \label{3.30} 
\end{equation}
or \[
\delta ^\nabla  \big(F^{\prime }(\frac 12||R ^\nabla
||^2)R^\nabla  \big) = 0\, .\]

${\rm(ii)}$ The Euler-Lagrangian equation for $p$-Yang-Mills functional $\mathcal{YM}_p\, , p \ge 2$ is
\begin{equation}
\delta ^\nabla ( ||R
^\nabla ||^{p-2}R ^\nabla ) =0\label{3.31}
\end{equation}
 \[\qquad \text{or}\qquad ||R
^\nabla ||^{p-2}\delta ^\nabla R
^\nabla - i_ {\text{grad} (||R
^\nabla ||^{p-2})}R
^\nabla =0\, .
\]}
\eqref{3.30} is also due to C. Gherghe (\cite {G}).
\smallskip

\begin{corollary}  Let $||R ^\nabla || =$ constant. Then the following are equivalent:
\begin{equation}
\aligned
{\rm(i)} \quad&\operatorname{A}\,  \operatorname{curvature}\, \operatorname{tensor}\,  R ^\nabla \text{is}\, \text{a}\,
\text{normalized}\, \text{exponential}\, \text{Yang-Mills}\, \,
\text{field}\, .\\
{\rm(ii)} \quad &\operatorname{A}\,  \operatorname{curvature}\, \operatorname{tensor}\,  R ^\nabla \text{is}\, \text{a}\,
\text{Yang-Mills}\, 
\text{field}\, .\\
{\rm(iii)} \quad  & \operatorname{A}\,  \operatorname{curvature}\, \operatorname{tensor}\,  R ^\nabla \text{is}\, \text{a}\, \, p-\text{Yang-Mills}\, \text{field}, p \ge 2\,\, .\\
{\rm(iv)} \quad  & \operatorname{A}\,  \operatorname{curvature}\, \operatorname{tensor}\,  R ^\nabla \text{is}\, \text{an}\,
\text{exponential}\, \text{Yang-Mills}\, 
\text{field}\, .\\
{\rm(v)} \quad  & \operatorname{A}\,  \operatorname{curvature}\, \operatorname{tensor}\,  R ^\nabla \text{is}\, \text{an}\, \,
F\text{-Yang-Mills}\, 
\text{field}\, .
\endaligned\label{3.32}
\end{equation}
\label{C: 3.8} 
\end{corollary}

\begin{proof}
This follows at once from \eqref{3.29}-\eqref{3.31}.
\end{proof}

\begin{lemma} 
 Let $S_{e,\mathcal{YM}^0 }$ and $S_{e,\mathcal{YM} }$ be the
stress-energy tensors defined by \eqref{3.9} and \eqref{3.10} respectively, then for any vector field
$X$ on $M$, we have
\begin{equation}
\begin{aligned} (\operatorname{div} S_{e,\mathcal{YM}^0  })(X)&=(\operatorname{div} S_{e,\mathcal{YM}  })(X)\\
& = \exp (\frac{||R^\nabla||^2}2)\langle\delta ^\nabla R^\nabla ,i_X R^\nabla \rangle+\exp (\frac{||R^\nabla||^2}2)\langle i_Xd^\nabla R^\nabla ,R^{\nabla} \rangle
\\
&\qquad - \langle i_{\text{grad}(\exp(\frac{||R^\nabla ||^2}2))}R^\nabla
,i_X R^\nabla \rangle\, ,
\end{aligned}
\label{3.33}
\end{equation}
where $\text{grad} \, (\, \bullet\, )$ is the gradient vector field of $\, \bullet\, .$
\label{L: 3.9}
\end{lemma}
\smallskip

\begin{definition}  
A curvature tensor $R^{\nabla} \in
A^2 \big (Ad(P) \big )$ is said to satisfy an
\emph {$e$-conservation law}  if $S_{e,\mathcal{YM}^0}$ is divergence free, i.e., 
\begin{equation}
\text{div} S_{e,\mathcal{YM}^0 } = \text{div} S_{e,\mathcal{YM} } = 0\, .
\label{3.34}
\end{equation}
\end{definition}
\smallskip

\begin{theorem}  Every normalized exponential Yang-Mills
field or every exponential Yang-Mills
field $R ^\nabla $ satisfies an $e$-conservation law.\label{T: 3.11}
\end{theorem}

\begin{proof} It is known that $R ^\nabla $ satisfies the
Bianchi identity
\begin{equation}
d^\nabla R ^\nabla =0\, .  \label{3.35}
\end{equation}
Therefore, by Corollary \ref{C: 3.7}, Lemma \ref{L: 3.9} and \eqref{3.35}, we immediately
derive the desired \eqref{3.34}.
\end{proof}

\section{Comparison theorems in Riemannian geometry}

In this section, we will discuss comparison theorems with applications on
Cartan-Hadamard manifolds or more generally on complete manifolds with a pole. We
recall a  Cartan-Hadamard manifold is a complete simply-connected Riemannian
manifold of nonpositive sectional curvature.
A {\it pole} is a point $x_0\in M$ such that the exponential map from the tangent space to $M$ at $x_0$ into $M$ is a diffeomorphism. By the
{\it radial curvature} $K $ of a manifold with a pole, we mean the
restriction of the sectional curvature function to all the planes which contain the unit vector $\partial (x)$ in $T_{x}M$ tangent to the unique geodesic joining $%
x_{0}$ to $x$ and pointing away from $x_{0}.$ Let the tensor $g - dr \bigotimes dr = 0$  on the radial direction $\partial $, and is just the
metric tensor $g$ on the orthogonal complement $\partial ^{\bot}$.

\begin{theorem}(Hessian comparison theorem \cite{GW, DW, HLRW, W11})  Let $(M,g)$ be a complete
Riemannian manifold with a pole $x_0$. Denote by $K(r)$ the radial
curvature of $M$. Then
\begin{equation}
   -\alpha ^2\leq K(r)\leq -\beta ^2\quad \text{with}\quad \alpha >0, \,  \beta
>0\tag{i \cite {GW}}
\end{equation}
\begin{equation}
\Rightarrow\quad\beta \coth (\beta r)\big (g-dr\otimes dr\big )\leq Hess(r)\leq \alpha \coth
(\alpha r)\big (g-dr\otimes dr\big );\label{4.1}
\end{equation}
\begin{equation*}
K(r) = 0 
\tag{ii \cite {GW}}
\end{equation*}
\begin{equation}
\Rightarrow\quad\frac 1r\big (g-dr\otimes dr\big ) = Hess(r);\label{4.2}
\end{equation}
\begin{equation*}
-\frac A{(1+r^2)^{1+\epsilon}}\leq K(r)\leq \frac B{(1+r^2)^{1+\epsilon}}\quad \text{with}\quad \epsilon > 0, \, A \ge 0, \quad \text{and}\quad 0 \le B < 2\epsilon\, 
\tag{iii \cite {GW}, \cite [Lemma 4.1.(iii)]{DW}}
\end{equation*}
\begin{equation}
\Rightarrow\quad \frac{1-\frac B{2\epsilon}}{r}\big (g-dr\otimes dr\big )\leq Hess(r)\leq \frac{e^{\frac {A}{2\epsilon}}}{r}\big (g-dr\otimes dr\big );
\label{4.3}
\end{equation}
\begin{equation*}-\frac {A}{r^2}\leq K(r)\leq -\frac {A_1}{r^2}\quad  \text{with}\quad  0 \le A_1 \le A\, 
\tag{iv \cite {HLRW}, \cite [Theorem A] {W11} }
\end{equation*}
\begin{equation}
\Rightarrow\quad\frac{1+\sqrt{1+4A_1}}{2r}\bigg(g-dr\otimes dr\bigg) \le  \text{Hess} (r) \le \frac{1+\sqrt{1+4A}}{2r}\bigg(
g-dr\otimes dr\bigg);\label{4.4}
\end{equation}
\begin{equation*}
- \frac {A(A-1)}{r^2}\le K(r) \le - \frac {A_1(A_1-1)}{r^2}\quad \text{with}\quad  A \ge A_1 \ge 1\, 
\tag{v \cite [Corollary 3.1] {W11}}
\end{equation*}
\begin{equation}
\Rightarrow\quad  \frac{A_1}{r}\bigg(
g-dr\otimes dr\bigg) \le \text{Hess} r  \leq \frac{A}{r}\bigg(
g-dr\otimes dr\bigg);
\label{4.5}
\end{equation}
\begin{equation*}
\frac {B_1(1-B_1)}{r^2}\leq K(r)\le \frac {B(1-B)}{r^2}\, ,\text{with}\quad 0 \le B, \, B_1 \le 1\, 
\tag{vi \cite [Corollary 3.5] {W11}}
\end{equation*}
\begin{equation}
\Rightarrow\quad\frac {|B - \frac 12| + \frac 12}{r} \bigg(g-dr\otimes dr\bigg)    \le \text{Hess} r  \le \frac{1+\sqrt{1+4B_1(1-B_1)}}{2r}\bigg(
g-dr\otimes dr\bigg);\label{4.6}
\end{equation}
\begin{equation*}
\frac {B_1}{r^2} \le K(r) \le \frac {B}{r^2}\quad \text{with}\quad  0 \le B_1 \le B \le \frac 14\, 
\tag{vii \cite [Theorem 3.5]{W11}}
\end{equation*}
\begin{equation}
\Rightarrow\quad \frac{1+\sqrt{1-4B}}{2r} \bigg(
g-dr\otimes dr\bigg) \le  \text{Hess} r  \le \frac{1+\sqrt{1+4B_1}}{2r}\bigg(g-dr\otimes dr\bigg);\label{4.7}
\end{equation}
\begin{equation*}
 -Ar^{2q}\leq K(r)\leq -Br^{2q}\quad \text{with}\quad A\geq B>0\, ,
q>0\, 
\tag{viii \cite {GW}}\end{equation*}
\begin{equation}
\begin{aligned}
\Rightarrow\quad & B_0r^q\big (g-dr\otimes dr\big )\leq Hess(r)\leq (\sqrt{A}\coth \sqrt{A}
)r^q\big (g-dr\otimes dr\big )\, ,
\operatorname{for}\,  r\geq 1\, , 
\end{aligned} \label{4.8}
\end{equation}
where \begin{equation} B_0=\min
\{1,-\frac{q+1}2+\big (B+(\frac{q+1}2)^2\big )^{\frac{1}{2}}\}\, .
 \label{4.9}
\end{equation}
\label{T: 4.1}
\end{theorem}

\begin{proof}$(\operatorname{i})$, $(\operatorname{ii})$ and $(\operatorname{viii})$ are treated in section 2 of \cite {GW}, $(\operatorname{iii})$ is proved in \cite {DW}, $(\operatorname{iv})$ is derived in \cite{HLRW, W11}, $(\operatorname{v})$ - $(\operatorname{vii})$ are proved in \cite {W11}.
\end{proof}

We note there are many applications of this Theorem  (cf., e.g., \cite {WW}), $(\operatorname{iv})$ extends the asymptotic comparison theorem in (\cite {GW}, \cite {PRS}, p.39), and $(\operatorname{vii})$ generalizes (\cite {EF}, Lemma 1.2 (b)). 
\smallskip

Let $\flat $ denote the bundle isomorphism that identifies the vector field $X$ with the differential one-form $X^{\flat}$, and let $\nabla $ be the Riemannian connection of $M$. 
Then the covariant derivative $\nabla X^{\flat}$ of $X^{\flat}$ is a $(0,2)$-type tensor, given by 
\begin{equation} \nabla X^{\flat} (Y,Z) = \nabla _Y X^{\flat} Z = \langle \nabla _Y X, Z\rangle\, , \quad \forall\, X,Y \in \Gamma (M)\, . \label{4.10}
\end{equation}
If $X$ is conservative, then 

\begin{equation} X = \nabla f,\quad  X^{\flat} = df\quad \text{and}\quad  \nabla X^{\flat} = \text{Hess} (f)\, .\label{4.11}
\end{equation}
for some scalar potential $f$ $($cf. \cite {CW}, p. 1527$)$.
 A direct computation yields $($cf., e.g., \cite {DW}$)$\begin{equation} \text{div}
(i_X S_{e,\mathcal{YM}^0 }) = \langle S_{e,\mathcal{YM}^0 },\nabla X^{\flat}\rangle  +  (\text{div}
S_{e,\mathcal{YM}^0 }) (X)\, ,\quad \forall\, X\in \Gamma (M)\, . \label{4.12}
\end{equation} By Theorem \ref{T: 3.11}, every normalized exponential Yang-Mills
field $R ^\nabla $ satisfies an $e$-conservation law.
It follows from the divergence theorem that for every bounded domain $D$ in $M\, $ with $C^1$ boundary $\partial D\, ,$
\begin{equation}
\int_{\partial D}S_{e,\mathcal{YM}^0 }(X,\nu ) ds_g = \int_D\langle S_{e,\mathcal{YM}^0 },\nabla X^{\flat}\rangle dv_g\, , \label{4.13}
\end{equation}
where $\nu $ is unit outward normal vector field along $\partial
D$ with $(n-1)$-dimensional volume element $ds_g$. When we choose scalar potential $f (x) = \frac 12 r^2 (x)$, 
\eqref{4.11} becomes 
\begin{equation} X = r\nabla r,\quad  X^{\flat} = r dr\quad \text{and}\quad  \nabla X^{\flat} = \text{Hess} (\frac 12 r^2) = dr\otimes dr + r\text{Hess} (r)\, .\label{4.14}
\end{equation}
The conservative vector field $X$ and $e$-conservation law will illuminate that the curvature of the base manifold $M$ via Hessian Comparison Theorems \ref{T: 4.1} influences the behavior of the stress energy tensor $S_{e,\mathcal{YM}^0 }$ and the behavior of the underlying criticality - curvature field $R ^\nabla \in A^2(Ad(P))$ with the help from the following concept \eqref{4.15} and estimate \eqref{4.20}. 

Analogous to $F$-degree, we introduce 

\begin{definition} For a given curvature field $R^\nabla$, 
the \emph{$e$-degree} $d_e$ is the quantity, given by \begin{equation}
d_e = \sup_{x \in M} \frac {\exp\big (\frac{||R^\nabla||^2}2 (x)\big )}{\exp\big (\frac{||R^\nabla||^2}2(x)\big ) - 1}\, .\label{4.15}
\end{equation}
\label{D: 4.2}
\end{definition} 

The $e$-degree $d_e$ will play a role in connecting two separated parts of the normalized stress-energy tensor $S_{e,\mathcal{YM}^0 }$. Since $\frac {e^t}{e^t - 1} $ is a decreasing function, with $1 \le \frac {e^t}{e^t - 1} \le \infty$, we have

\begin{proposition} Suppose  \begin{equation}  \frac{||R^\nabla||^2}2(x) \le c\ \  \forall\ \  x \in M\, ,\label{4.16}
\end{equation} where $c > 0$ is a constant.  Then

\begin{equation} d_e \ge \frac {e^c}{e^c - 1}\, .
\label{4.17}
\end{equation}\label{P: 4.3}
\end{proposition}

\begin{lemma} Let $M$ be a complete $n$-manifold with a pole $x_0$.
Assume that there exist two positive functions $h_1(r)$ and
$h_2(r)$ such that
\begin{equation}
h_1(r)(g-dr\otimes dr)\leq Hess(r)\leq h_2(r)(g-dr\otimes dr)
\label{4.18}
\end{equation}
on $M\backslash \{x_0\}$. If $h_2(r)$ satisfies
\begin{equation}
rh_2(r)\geq 1\, ,  \label{4.19}
\end{equation}
and $||R^\nabla|| > 0$ on $M\, ,$
then
\begin{equation}
\langle S_{e,\mathcal{YM}^0 },\nabla X^\flat \rangle\,\geq
\,\big(1+(n-1)rh_1(r)-2 d_e ||R^\nabla||_{\infty}^2rh_2(r)\big)\big (\exp (\frac{||R^\nabla||^2}2) - 1\big )\, ,  \label{4.20}
\end{equation}
where $X=r \nabla  r$.\label{L: 4.4}
\end{lemma}

\begin{proof}Choose an orthonormal frame $\{e_i,\frac
\partial
{\partial r}\}_{i=1,...,n-1}$ around $x\in M\backslash \{x_0\}$.
Take $ X=r\nabla r$. Then
\begin{equation}
\nabla _{\frac \partial {\partial r}}X=\frac \partial {\partial r}\, ,
\label{4.21}
\end{equation}
\begin{equation}
\nabla _{e_i}X=r\nabla _{e_i}\frac \partial {\partial
r}=rHess(r)(e_i,e_j)e_j\, . \label{4.22}
\end{equation}
Using \eqref{3.10}, \eqref{4.14}, $\big ($or \eqref{4.21}, \eqref{4.22}$\big )$, we have
\begin{equation}
\begin{aligned} \langle S_{e,\mathcal{YM}^0 },\nabla X^\flat \rangle&=\big (\exp (\frac{||R^\nabla||^2}2) - 1\big )(1+%
\sum_{i=1}^{n-1}rHess(r)(e_i,e_i)) \\
&\qquad -\sum_{i,j=1}^{n-1}\exp (\frac{||R^\nabla||^2}2)\langle i_{e_i} R^{\nabla} ,i_{e_j} R^{\nabla} \rangle r Hess(r)(e_i,e_j) \\
&\qquad -\exp (\frac{||R^\nabla\|^2}2)
\langle i_{\frac \partial {\partial r}} R^{\nabla} ,i_{\frac \partial {\partial r}} R^{\nabla} \rangle\, .
\end{aligned}
\label{4.23}
\end{equation}
By \eqref{4.18} and \eqref{4.15}, 
\eqref{4.23} implies that
\begin{equation}
\begin{aligned} & \langle S_{e,\mathcal{YM}^0 },\nabla X^\flat \rangle\\
&\geq \big (\exp (\frac{||R^\nabla||^2}2) - 1\big )\big(1+(n-1)rh_1(r)\big)
\\
&\qquad -\big (\exp (\frac{||R^\nabla||^2}2) - 1\big )\sum_{i=1}^{n-1}\langle i_{e_i} R^{\nabla} ,i_{e_i} R^{\nabla} \rangle r h_2(r) \frac{\exp (\frac{||R^\nabla||^2}2) }{\big (\exp (\frac{||R^\nabla||^2}2) - 1\big )}\\
&\qquad -\big (\exp (\frac{||R^\nabla||^2}2) - 1\big )\langle i_{\frac \partial {\partial r}} R^{\nabla} ,i_{\frac \partial {\partial r}} R^{\nabla} \rangle \frac{\exp (\frac{||R^\nabla||^2}2) }{\big (\exp (\frac{||R^\nabla||^2}2) - 1\big )}\\
&\geq \big (\exp (\frac{||R^\nabla||^2}2) - 1\big )\big(1+(n-1)rh_1(r)-2||R^\nabla||^2rh_2(r) \frac{\exp (\frac{||R^\nabla||^2}2) }{\big (\exp (\frac{||R^\nabla||^2}2) - 1\big )}\big) \\
&\qquad +\big (\exp (\frac{||R^\nabla||^2}2) - 1\big )(rh_2(r)-1)\langle i_{\frac\partial{\partial r}} R^{\nabla}, i_{\frac{\partial}{\partial r}}R^{\nabla} \rangle\frac{\exp (\frac{||R^\nabla||^2}2) }{\big (\exp (\frac{||R^\nabla||^2}2) - 1\big )}\\
&\geq \big(1+(n-1)rh_1(r)-2||R^\nabla||^2rh_2(r) d_e\big) \big (\exp (\frac{||R^\nabla||^2}2) - 1\big )\, .\end{aligned} \label{4.24}
\end{equation}
The last two steps follow from \eqref{4.19} and the fact that
\begin{equation}
\aligned & \sum_{i=1}^{n-1}\langle i_{e_i} R^{\nabla} ,i_{e_i} R^{\nabla} \rangle+\langle i_{\frac \partial {\partial r}} R^{\nabla} ,i_{\frac \partial {\partial r}} R^{\nabla} \rangle \\
 &\quad = \sum_{1\le j_1\le n}\sum_{i=1}^{n}\langle R^\nabla (e_i,e_{j_1}), R^\nabla (e_i,e_{j_1})\rangle
\\
&\quad = 2 ||R^\nabla||^2\, ,
\endaligned\label{4.25}
\end{equation}
where $e_n = \frac
\partial{\partial r}\, .$
Now the Lemma follows immediately from \eqref{4.24} and \eqref{3.7}.
\end{proof}

\section{Monotonicity formulae}

In this section, we will establish monotonicity formulae on
complete manifolds with a pole.\smallskip

\begin{theorem}[Monotonicity formulae]\label{t5.1} Let $(M,g)$ be an $n-$dimensional complete
Riemannian manifold with a pole $ x_0$, $Ad(P)$ be the adjoint bundle and the curvature tensor $ R^\nabla \in A^2\big (Ad(P)\big )$ be an exponential Yang-Mills
field.
Assume that the radial curvature $K(r)$ of $M$ and the curvature tensor $ R^\nabla$ satisfy one of the
following seven conditions:\smallskip

\begin{equation}
\aligned
{\rm(i)}&\quad -\alpha ^2\leq K(r)\leq -\beta ^2\, \text{with}\, \alpha >0, \beta
>0\, \text{and}\,  (n-1)\beta - 2d_e\alpha \|R^\nabla \|^2 _\infty \geq 0;\\
{\rm(ii)}&\quad K(r) = 0\, \text{with}\,    n-2d_e\|R^\nabla \|^2 _\infty>0;\\
{\rm(iii)}&\quad -\frac A{(1+r^2)^{1+\epsilon}}\leq K(r)\leq \frac B{(1+r^2)^{1+\epsilon}} \text{with}\, \epsilon > 0\, , A \ge 0\, , 0 < B < 2\epsilon\, , \text{and}\\
&\qquad n - (n-1)\frac B{2\epsilon} -2d_e e^{\frac {A}{2\epsilon}}\|R^\nabla \|^2 _\infty > 0;\\
{\rm(iv)}&\quad-\frac {A}{r^2}\leq K(r)\leq -\frac {A_1}{r^2}\,  \text{with}\quad  0 \le A_1 \le A\, ,\text{and}\\
&\qquad 1 + (n-1)\frac{1 + \sqrt {1+4A_1}}{2} - d_e (1 + \sqrt {1+4A})\|R^\nabla \|^2 _\infty > 0;\\
{\rm(v)}&\quad- \frac {A(A-1)}{r^2}\le K(r) \le - \frac {A_1(A_1-1)}{r^2}\, \text{and}\,   A \ge A_1 \ge 1\, , \text{and}\\ 
&\qquad 1+(n-1)A_1-2d_eA\|R^\nabla \|^2 _\infty > 0;\\
{\rm(vi)}&\quad \frac {B_1(1-B_1)}{r^2}\leq K(r)\le \frac {B(1-B)}{r^2}\, ,\text{with}\quad 0 \le B, \, B_1 \le 1\, ,  \text{and} \\
&\qquad 1 + (n-1)(|B-\frac {1}{2}|+ \frac {1}{2}) -d_e\big (1 + \sqrt {1+4B_1(1-B_1)}\big )\|R^\nabla \|^2 _\infty > 0;\\
{\rm(vii)}&\quad \frac {B_1}{r^2} \le K(r) \le \frac {B}{r^2} \text{with}\quad  0 \le B_1 \le B \le \frac 14\, , \text{and}\\ 
 &\qquad 1+ (n-1)\frac{1 + \sqrt {1-4B}}{2} -d_e(1 + \sqrt {1+4B_1} )\|R^\nabla \|^2 _\infty   > 0. 
\endaligned
\label{5.1}
\end{equation} 

Then 

\begin{equation}
\frac 1{\rho _1^\lambda }\int_{B_{\rho _1}(x_0)}\big (\exp (\frac{||R^\nabla||^2}2)-1\big)\, dv \leq \frac 1{\rho _2^\lambda }\int_{B_{\rho
_2}(x_0)}\big (\exp (\frac{||R^\nabla||^2}2)-1\big)\, dv\, ,\label{5.2}
\end{equation}
for any $0<\rho _1\leq \rho _2$, where
\begin{equation}
\lambda \le \begin{cases} n-2d_e\frac \alpha \beta \|R^\nabla \|^2 _\infty\ &\text {if } K(r)\  \text{obeys $($i$)$}\, ,\\
n-2d_e\|R^\nabla \|^2 _\infty\  &\text {if } K(r)\ \text{obeys $($ii$)$}\, ,\\
n - (n-1)\frac B{2\epsilon} -2d_ee^{\frac {A}{2\epsilon}}\|R^\nabla \|^2 _\infty\  &\text{if } K(r) \text{
obeys $($iii$)$}\, ,\\
1 + (n-1)\frac{1 + \sqrt {1+4A_1}}{2} - d_e(1 + \sqrt {1+4A})\|R^\nabla \|^2 _\infty\  &\text{if } K(r) \text{
obeys $($iv$)$}\, ,\\
1+(n-1)A_1-2d_eA\|R^\nabla \|^2 _\infty\ &\text{if } K(r) \text{
obeys $($v$)$}\, ,\\
1 + \frac{n-1}{(|B-\frac {1}{2}|+ \frac {1}{2})^{-1}} -d_e\big (1 + \sqrt {1+4B_1(1-B_1)}\big )\|R^\nabla \|^2 _\infty\ &\text{if } K(r)\ \text
{obeys $($vi$)$}\, ,\\
1+ (n-1)\frac{1 + \sqrt {1-4B}}{2} -d_e(1 + \sqrt {1+4B_1} )\|R^\nabla \|^2 _\infty\ &\text{if } K(r)\ \text{obeys $($vii$)$}\, .\end{cases}  
\label{5.3}
\end{equation}\label{T: 5.1}
\end{theorem}

\begin{proof}
Take a smooth vector field $X=r\nabla r$ on $M\, .$ If $K(r)$ satisfies $($i$)$, then by Theorem \ref{4.1} and the increasing function $\alpha r \coth (\alpha r)\to 1\, $ as $r \to 0\, ,$ \eqref{4.19} holds. Now Lemma \ref{4.1} is applicable and by \eqref{4.20}, we have on $B_\rho(x_0)\backslash\{x_0\}\, ,$ for every $\rho >0,$
\begin{equation}
\begin{aligned}
 &\langle S_{e,\mathcal{YM}^0 },\nabla X^\flat \rangle\\
  &\geq
\,\big(1+(n-1)\beta r \coth (\beta r)-2 d_e \alpha r \coth
(\alpha r)\|R^\nabla \|^2 _\infty\big)\big (\exp (\frac{||R^\nabla||^2}2)-1\big)\\ 
&=\,\big(1+ \beta r \coth (\beta r)(n-1-2 \cdot d_e \cdot \frac{\alpha r \coth
(\alpha r)}{\beta r \coth (\beta r)}\|R^\nabla \|^2 _\infty)  \big)\big (\exp (\frac{||R^\nabla||^2}2)-1\big)\\
& > \,\big(1+ 1 \cdot (n-1-2\cdot d_e\cdot \frac{\alpha }{\beta } \cdot 1\|R^\nabla \|^2 _\infty)  \big)\big (\exp (\frac{||R^\nabla||^2}2)-1\big)\\
&\ge \lambda \big (\exp (\frac{||R^\nabla||^2}2)-1\big)\, ,
\end{aligned}\label{5.4}
\end{equation}
provided that $$n-1-2\cdot d_e \cdot \frac{\alpha }{\beta }\|R^\nabla \|^2 _\infty    \ge 0\, ,$$ since $$\beta r \coth
(\beta r) > 1\ \text{for}\ r > 0\ , \text{and}\ \frac{\coth
(\alpha r)}{\coth (\beta r)} < 1\, ,\text{for}\ 0 < \beta < \alpha\, ,
$$ 
and $\coth$ is a decreasing function. Similarly, from Theorem \ref{T: 4.1} and Lemma \ref{L: 4.4}, the above inequality holds for the cases (ii) - (vii) on $B_\rho(x_0)\backslash\{x_0\}\, .$  Thus, by the continuity of $\langle S_{e,\mathcal{YM} },\nabla X^\flat \rangle$ and $ \exp (\frac{||R^\nabla||^2}2)\, ,$ and \eqref{3.10}, we
have for every $\rho >0,$
\begin{equation}
\aligned
&\langle S_{e,\mathcal{YM}^0 },\nabla X^\flat \rangle\geq \lambda
\big (\exp (\frac{||R^\nabla||^2}2)-1\big)  \qquad \text{in} \quad B_\rho(x_0) \\
& \rho\, \, \big (\exp (\frac{||R^\nabla||^2}2)-1\big)  )\geq S_{e,\mathcal{YM}^0 }(X,\frac \partial {\partial r}) \qquad \text{on} \quad \partial B_\rho(x_0)\, . \endaligned\label{5.5}
\end{equation}
It follows from \eqref{4.13} and \eqref{5.5} that
\begin{equation}
\rho\int_{\partial B_\rho(x_0)}\big (\exp (\frac{||R^\nabla||^2}2)-1\big)\,   ds \geq \lambda
\int_{B_\rho(x_0)}\big (\exp (\frac{||R^\nabla||^2}2)-1\big)\, dv\, . \label{5.6}
\end{equation}
Hence, we get from \eqref{5.6} the following
\begin{equation}
\frac{\int_{\partial B_\rho(x_0)}\big (\exp (\frac{||R^\nabla||^2}2)-1\big )\, ds}{\int_{B_\rho(x_0)}
\big (\exp (\frac{||R^\nabla||^2}2)-1\big )\, dv} \geq \frac \lambda \rho\, .  \label{5.7}
\end{equation}
The coarea formula implies that
\begin{equation}
\frac {d}{d\rho}\int_{B_\rho(x_0)}\big (\exp (\frac{||R^\nabla||^2}2)-1\big )\, dv =\int_{\partial B_\rho(x_0)}
\big (\exp (\frac{||R^\nabla||^2}2)-1\big)\, ds\, .\label{5.8}
\end{equation}
Thus we have
\begin{equation}
\frac{\frac {d}{d\rho}\int_{B_\rho(x_0)}\exp (\frac{||R^\nabla||^2}2)-1\, dv}{\int_{B_\rho(x_0)}\exp (\frac{||R^\nabla||^2}2)-1\, dv}\\
\geq \frac \lambda
\rho
\label{5.9}
\end{equation}
for a.e. $\rho >0\, .$ By integration \eqref{5.9} over $[\rho _1,\rho _2]$, we have
\begin{equation}
\ln \int_{B_{\rho _2}(x_0)}\big (\exp (\frac{||R^\nabla||^2}2)-1\big )\, dv -\ln
\int_{B_{\rho _1}(x_0)}\big (\exp (\frac{||R^\nabla||^2}2)-1\big )\, dv \geq \ln \rho
_2^\lambda -\ln \rho _1^\lambda\, . \label{5.10}
\end{equation}
This proves \eqref{5.2}.
\end{proof}

\begin{corollary} Suppose that $M$ has constant sectional
curvature $-\alpha ^2\le 0$ and
$$
\begin{cases}
n-1-2d_e \|R^\nabla \|^2 _\infty \ge 0 \ & \text{if} \  \ \alpha \ne 0;\\
n - 2d_e\|R^\nabla \|^2 _\infty > 0\ &\text{if}\ \ \alpha=0.
\end{cases}
$$
Let
$R^\nabla \in A^2\big (Ad(P)\big )$ be an exponential Yang-Mills
field. Then
\begin{equation}
\begin{aligned} 
&\frac 1{\rho _1^{n-2d_e\|R^\nabla \|^2 _\infty}}\int_{B_{\rho _1}(x_0)}\big (\exp (\frac{||R^\nabla||^2}2)-1\big ) \, dv \\
&\leq \frac 1{\rho _2^{n-2d_e\|R^\nabla \|^2 _\infty}} \int_{B_{\rho
_2}(x_0)}\big (\exp (\frac{||R^\nabla||^2}2)-1\big )\,  dv\, ,
\end{aligned}\label{5.11}
\end{equation}
for any $x_0\in M$ and $0<\rho _1\leq \rho _2$.\label{C: 5.2}
\end{corollary}

\begin{proof} In Theorem \ref{T: 5.1}, if we take $\alpha =\beta \ne 0$ for
the case (i) or $\alpha=0$ for the case (ii), this corollary follows immediately. \end{proof}

\begin{proposition} Let $(M,g)$ be an $n-$dimensional
complete Riemannian manifold whose radial curvature satisfies

\begin{equation}
{\rm(viii)} -Ar^{2q}\leq K(r)\leq -Br^{2q}\quad \text{with}\quad  A\geq B>0\quad \text{and}\quad q>0.\label{5.12}
\end{equation}

Let $R^\nabla$ be an exponential Yang-Mills field, and 
\begin{equation}
\delta
:=(n-1)B_0-2d_e\|R^\nabla \|^2 _\infty\sqrt{A}\coth \sqrt{A} \geq 0\, ,\label{5.13}
\end{equation}
where $B_0$ is as in \eqref{4.9}.
Suppose that \eqref{5.18} holds. Then
\begin{equation}
\begin{aligned}
&\frac 1{\rho _1^{1 + \delta}}\int_{B_{\rho
_1}(x_0)-B_1(x_0)}\big (\exp (\frac{||R^\nabla||^2}2)-1\big ) \, dv \\
&\leq \frac 1{\rho
_2^{1 + \delta}}\int_{B_{\rho
_2}(x_0)-B_1(x_0)}\big (\exp (\frac{||R^\nabla||^2}2)-1\big ) \, dv\, ,
\end{aligned}\label{5.14}
\end{equation}
for any $1\leq \rho _1\leq \rho _2$.\label{P: 5.3}
\end{proposition}

\begin{proof} Take $X=r\nabla r$. Applying Theorem \ref{T: 4.1}, \eqref{4.19}, and \eqref{4.20}, we have
\begin{equation}
\aligned \langle S_{e,\mathcal{YM}^0 },\nabla X^\flat \rangle &\geq \big (\exp (\frac{||R^\nabla||^2}2)-1\big)(1+\delta r^{q+1}\big)
\endaligned
\label{5.15}
\end{equation}
and
\begin{equation}
\aligned & S_{e,\mathcal{YM}^0 }(X,\frac \partial {\partial r}) = \exp (\frac{||R^\nabla||^2}2) \big ( 1 - \langle i_{\frac\partial{\partial r}}R^\nabla, i_{\frac\partial{\partial r}}R^\nabla \rangle \big ) - 1\qquad \text{on} \quad \partial B_1(x_0) \\
&  S_{e,\mathcal{YM}^0 }(X,\frac \partial {\partial r}) = \rho (\exp (\frac{||R^\nabla||^2}2)  \big (1 - \langle i_{\frac\partial{\partial r}}R^\nabla, i_{\frac\partial{\partial r}}R^\nabla \rangle \big ) - \rho\qquad \text{on} \quad \partial B_\rho(x_0)\, . \endaligned
\label{5.16}
\end{equation}

It follows from \eqref{4.13} that

\begin{equation} \aligned & \rho\int_{\partial B_\rho(x_0)} 
\exp (\frac{||R^\nabla||^2}2) \big (1 - \langle i_{\frac\partial{\partial r}}R^\nabla, i_{\frac\partial{\partial r}}R^\nabla \rangle \big ) - 1\,  ds \\
& \qquad - \int_{\partial B_1(x_0)} \exp (\frac{||R^\nabla||^2}2) \big (1 - \langle i_{\frac\partial{\partial r}}R^\nabla, i_{\frac\partial{\partial r}}R^\nabla \rangle \big ) - 1\,  ds\\
& \ge  \int_{B_\rho(x_0)-B_1(x_0)} (1+\delta
r^{q+1}) \big (\exp (\frac{||R^\nabla||^2}2)-1\big)\, .
\endaligned\label{5.17}
\end{equation}
Whence, if
\begin{equation}
\int_{\partial B_1(x_0)} \exp (\frac{||R^\nabla||^2}2) \big (1 - \langle i_{\frac\partial{\partial r}}R^\nabla, i_{\frac\partial{\partial r}}R^\nabla \rangle \big ) - 1\,  ds \ge 0\, ,\label{5.18}\end{equation}
then
\begin{equation}
\rho\int_{\partial B_\rho(x_0)} \big (\exp (\frac{||R^\nabla||^2}2)-1\big ) \, ds \geq (1+\delta
)\int_{B_\rho(x_0)-B_1(x_0)}\big (\exp (\frac{||R^\nabla||^2}2)-1\big) \, dv\, , \label{5.19}\end{equation}
for any $\rho > 1\, .$
Coarea formula then implies
\begin{equation}
\frac{d\int_{B_\rho(x_0)-B_1(x_0)}\big (\exp (\frac{||R^\nabla||^2}2)-1\big ) \, dv}{
\int_{B_\rho(x_0)-B_1(x_0)}\big (\exp (\frac{||R^\nabla||^2}2)-1\big ) \, dv}\geq \frac{1+\delta} \rho d\rho
\label{5.20}
\end{equation}
for a.e. $\rho\geq 1$. Integrating \eqref{5.20} over $[\rho _1,\rho _2]$, we
get
\begin{equation}
\begin{aligned}
&\ln \big(\int_{B_{\rho _2}(x_0)-B_1(x_0)}\big (\exp (\frac{||R^\nabla||^2}2)-1\big )\, dv \big)\\
&\ \ -\ln
\big( \int_{B_{\rho _1}(x_0)-B_1(x_0)}\big (\exp (\frac{||R^\nabla||^2}2)-1\big )\, dv \big)  \\
&\geq  (1+\delta) \ln \rho _2- (1+\delta) \ln \rho _1\, .
\end{aligned}
\label{5.21}\end{equation}

Hence we prove the proposition.
\end{proof}

\begin{corollary}  Let $K(r)\, $ and $\delta$ be as in Proposition \ref{5.3}, satisfying \eqref{5.12} and \eqref{5.13} respectively, $\, R^\nabla $ be an exponential Yang-Mills field. Suppose
\begin{equation}
\exp (\frac{||R^\nabla||^2}2) \big (1 - \langle i_{\frac\partial{\partial r}}R^\nabla, i_{\frac\partial{\partial r}}R^\nabla \rangle \big ) \ge 1\label{5.22}
\end{equation}
on $\partial B_1\, .$ 
Then \eqref{5.14} holds. \label{C: 5.4}
\end{corollary}

\begin{proof}
The assumption \eqref{5.22} implies that \eqref{5.18} holds, and the assertion follows from Proposition \ref{P: 5.3}.
\end{proof}

\section{Vanishing theorems for exponential Yang-Mills fields}

\begin{theorem}[Vanishing Theorem]\label{T: 6.1} Suppose that the radial curvature $K(r)$ of $M$ satisfies one of the seven growth conditions in \eqref{5.1} $(\rm{i})$-$(\rm{vii}),$
Theorem \ref{T: 5.1}. Let $R^\nabla$ be an exponential Yang-Mills field satisfying the $\mathcal {YM}_e^0$-energy
functional growth condition
\begin{equation}
\int_{B_\rho(x_0)} \big (\exp(\frac{||R^\nabla ||^2}2) - 1\big )\, dv = o(\rho^\lambda )\quad \text{as
} \rho\rightarrow \infty\, ,\label{6.1}
\end{equation}
where $\lambda $ is given by \eqref{5.3}. Then $\exp(\frac{||R^\nabla ||^2}2)\equiv 1\, ,$ and hence $R^\nabla  \equiv 0$.  In particular, every exponential Yang-Mills field $R^{\nabla}$  with finite normalized exponential Yang-Mills $\mathcal {YM}_e^0$-energy
functional  vanishes on $M$.
\end{theorem}

\begin{proof} This follows at once from Theorem \ref{T: 5.1}. 
\end{proof}

\begin{proposition} Let $(M,g)$ be an $n-$dimensional
complete Riemannian manifold whose radial curvature satisfies
\eqref{5.12} $(\rm{viii})\, ,$ Proposition \ref{5.1}. 
Let $\delta$ be as in \eqref{5.13} in which $B_0$ is as in \eqref{4.9}.
Suppose \eqref{5.18} holds.  Then every exponential Yang-Mills field $R^\nabla$ with the growth
condition \begin{equation}
\int_{B_\rho(x_0)-B_1(x_0)}\big (\exp(\frac{||R^\nabla ||^2}2) - 1\big )\,  dv = o(\rho^{1 + \delta
})\text{\quad as }\rho\rightarrow \infty\label{6.2}\end{equation}
 vanishes on $M-B_1(x_0)\, ,$
In particular, if $R^\nabla
$ has finite normlalized exponential Yang-Mills energy on
$M-B_1(x_0)$, then $R^\nabla \equiv 0$ on
$M-B_1(x_0)$.\label{P: 6.2}
\end{proposition}

\begin{proof} This follows at once from Proposition \ref{P: 5.3}.
\end{proof}
\section{Vanishing theorems from exponential Yang-Mills fields to $F$-Yang-Mills fields}

\begin{theorem} Suppose that the radial curvature $K(r)$ of $M$ satisfies one of the seven growth conditions in \eqref{1.1} ${\operatorname(i)}-
{\operatorname(vii)}$, Theorem A , in which $d_F=1$.
Let $R^\nabla$ be an exponential Yang-Mills field with $||R^\nabla|| = $constant and  
\begin{equation}
\operatorname{Volume}\big (B_{\rho} (x_0)\big ) = o({\rho}^{\lambda} )\quad \text{as
} \rho\rightarrow \infty\, ,\label{7.1}
\end{equation}
where $\lambda $ is given by \eqref{1.4}, in which $d_F=1$. Then $R^\nabla  \equiv 0$.  In particular, every exponential Yang-Mills field $R^{\nabla}$  with constant $||R^\nabla||$ over manifold which has finite volume, $\operatorname{Volume}(M) < \infty$ vanishes.\label{T: 7.1}
\end{theorem}

\begin{proof} By Corollary \ref{C: 3.8}, this exponential Yang-Mills filed $R^\nabla$ is a Yang-Mills field which is a special case of $F$-Yang-Mills field, where $F$ is the identity map. Thus the F-degree of the identity map $d_F = 1\, .$
Now we apply $F$-Yang-Mills Vanishing Theorem A in which $F(t) =t$,  $d_F=1$, the $F$-Yang-Mills functional $\mathcal {YM}_F$ growth condition, \eqref{1.3} is transformed to the volume of the base manifold growth condition, \eqref{7.1}, and the conclusion $R^\nabla \equiv 0$ follows. 
\end{proof}

\section{An average principle, isoparametric and sobelov inequalities}   

In this section, we state, interpret, and apply an average principe in a simple discrete version, then extend it to a dual (or continuous) version: \smallskip

\begin{proposition} [{An average principle of concavity $(\operatorname{resp.}\,  convexity, linearity )$}]

\noindent
Let $f$ be a concave function $($resp. convex function, linear function $)$ .  Then
\begin{equation}
\begin{aligned} f(\text{average}) &\ge \text{average}\, (f)\, ,\\
\big (\operatorname{resp.}\quad  f(\text{average}) &\le \text{average}\, (f)\, ,\\
 f(\text{average}) & = \text{average}\, (f)\, \quad \big ). 
\end{aligned}
\label{8.1}
\end{equation}\label{P: 8.1}
\end{proposition} 

Applying \eqref{8.1}, where a convex function $f = \exp\, $ and ``average" is taken over two positive numbers with respect to the sum, yields
one of the simplest inequalities that has far-reaching impacts

\begin{equation}
\sqrt {a \cdot b}  = \overset  {``f(average)" \swarrow}  {\exp  (\frac {A+B}{2})} \overset {(\text{Average}\, \text{Principle})} \le  \overset  {``average(f)"\swarrow} {\frac {\exp A + \exp B}{2}} = \frac {a + b}{2}\, . 
\label{8.2}
\end{equation}
That is,

\begin{example}[G.M. $\le$  A.M.] The Geometric Mean is no greater than the Arithmetic Mean: 
\begin{equation} 
\begin{aligned}
\sqrt {a \cdot b} & \le  \frac {a + b}{2}\, , \quad \operatorname{for}\quad a, b > 0\\
 \operatorname{with}\, ``&="\,  \operatorname{holds}\, 
\operatorname{if}\,  \operatorname{and}\, \operatorname{only}\, \operatorname{if}\, a=b\, .\end{aligned}\label{8.3} \end{equation}
\end{example}

\noindent
Indeed, Let $a = \exp(A)$ and $b=\exp(B)\, .$ Then  applying An Average Principle of Convexity \eqref{8.1}, where $f = \exp\, $ yields \smallskip

\medskip

\noindent
{\bf A geometric interpretation of this inequality}:
 
Among all rectangles on the Euclidean plane with a given perimeter $\mathcal L$, the square has the largest area $\mathcal A$.\smallskip

By duality,  this means parallelly \smallskip

Among all rectangles on the Euclidean plane with a given area $\mathcal A$, the square has the least perimeter $\mathcal L$.

Indeed, \begin{equation} 16\, \mathcal A = 16\, a \cdot b \le  (2a +2b)^2\ = {\mathcal L}^2 \label{8..4}\end{equation}
Equality holds if and only if the rectangles are squares, i.e., $a=b$.\smallskip

A dual approach from discreteness to continuity yields\smallskip

\noindent{\bf A sharp isoperimetric inequlality for plane curves:} Among all simple closed smooth curves on the Euclidean plane with a given length $L$, the circle encloses the largest area $A\, .$

\begin{equation} 4\pi A \le L^2\label{8.5}\end{equation}
Equality holds if and only if the curve encloses a disk. 

This is equivalent to 
\smallskip

\noindent{\bf  The Sobolev inequality on $\mathbb {R}^2$ with optimal constant}: If $u \in W^{1,1}(\mathbb {R}^2)$, then
 
\begin{equation} 4 \pi \int _ {\mathbb {R}^2} |u|^{2}\, dx \le \bigg (\int _ {\mathbb {R}^2} |\nabla u|\, dx\bigg )^{2} .\label{8.6}\end{equation}\bigskip

Similarly, applying \eqref{8.1}, where $f = \exp\, $ and ``average" is averaging the sum of $n$ positive numbers, $n \ge 2$, yields

\begin{equation}
\root n\of{a_1 \cdot a_n} = \overset  {``f(average)" \swarrow}  {\exp  (n^{-1}\sum_{j=1}^nA_j)} \overset {(\text{Average}\, \text{Principle})} \le  \overset  {``average(f)"\swarrow}  {n^{-1}\sum_{j=1}^n\exp A_j} =n^{-1}\sum_{j=1}^na_j\, . 
\label{8.7}
\end{equation}
That is,

\begin{example}[The geometric mean of the numbers is no greater than the arithmetic mean of $n$ positive numbers]
\begin{equation} 
\begin{aligned} 
\root n \of{a_1 \cdot a_n}  &\le  \frac {a_1 + \cdots + a_n}{n}\, , \quad \text{for}\quad a_1, \cdots a_n > 0\, , \\
 \text{with}\, ``&="\,  \text{holds}\, 
\text{if}\,  \text{and}\, \text{only}\, \text{if}\, a_1= \cdots = a_n\, .\end{aligned} 
\label{8.8}
\end{equation}
\end{example}
For a dual version, let a concave function $f=\log$, An Average Principle, Proposition {8.1} yields 
\begin{example} 
Let $g$ be a nonnegative measurable function on $[0,1]$. Then 
\begin{equation} \log \int _0^1  g(t)\, dt \ge \int _0^1 \log \big ( g(t) \big )\, dt
\label{8.9}
\end{equation}
whenever the right side is defined.
\end{example}

Isoperimetric and Sobolev inequlalities can be generalized to higher dimensional Euclidean spaces.  As in dimension two, 
the $n$-dimensional sharp isoperimetric inequality is equivalent (for sufficiently smooth domains) to :\smallskip

\noindent
{\bf The Sobolev inequality on $\mathbb {R}^n$ with optimal constant}

\noindent
If $u \in W^{1,1}(\mathbb {R}^n)$ and $\omega_n$ is the volume of the unit ball in $\mathbb {R}^n\, $, then
 \begin{equation}
\bigg (\int _ {\mathbb {R}^n} |u|^{\frac {n}{n-1}}\, dx\bigg )^{\frac {n-1}{n}}\le  \frac {1}{n} \frac {1}{\sqrt[n]{\omega_n}}\int _ {\mathbb {R}^n} |\nabla u|\, dx.\label{8.10}\end{equation}

Isoperimetric and Sobolev inequlalities are extended to Riemannian manifolds $M$ with sharp constants 
and applications to optimal sphere theorems (cf., e.g., Wei-Zhu \cite {WZ}).

\begin{theorem}[ A sharp isoperimetric inequality \cite {Du, WZ}] For
{\bf every} domain $\Omega\, ($in $M)$, there exists a constant
$C(M)$ depending on $M$ 
such that 
\begin{equation} P^n \ge n^n \omega _n
V^{n-1}(1 - C(M) V^{\frac 2n}), \label{8.11} \end{equation} 
where
$P=vol(\partial \Omega), V=vol(\Omega)\, ,$ and  $\omega _n$ is the
volume of the unit ball in ${R}^n$. 

\noindent
Furthermore, on simply connected Riemannian manifolds of dimension $n$ with Ricci curvature bounded
from below by $n - 1$, the best $C(M)$ one can take in the above inequality \eqref{8.10} is greater than
or equal to \begin{equation}C_0 = \frac
{n(n-1)}{2(n+2) \omega_n^{\frac 2n}}\, . \label{8.12} \end{equation}
\label{T: 8.4}\end{theorem}

It is then by a standard technique, via coarea formula and
Cavalieri's principle, that \eqref{8.11} is equivalent to the
following:

\begin{theorem}[A sharp Sobolev inequality \cite{WZ}]   There exists a
constant $A=A(M)$ such that $\forall \varphi \in W^{1,1}(M),$ \begin{equation}
(\int _M|\varphi|^{\frac{n}{n-1}} dv)^{\frac {n-1}{n}} \le K(n,1)
(\int _M|\nabla \varphi| dv) +A(M)(\int _M|\varphi|^{\frac n{n+1}}
dv)^{\frac {n+1}{n}},\ \ \ \ \ \ \ \label{8.13}\end{equation}
 where $$K(n,1)=\lim_{p\to +1}K(n,p)=\frac 1{n \omega_n^{\frac 1n}}\, .$$
\end{theorem}

This isoperimetric inequality \eqref{8.11} certainly  has its roots
in global analysis and partial differential equations (see, e.g.,
\cite{AuL}). Furthermore, the optimal constants in \eqref{8.11} will
have some geometric and even topological applications. An
immediate example is that sharp estimate on $C(M)$ recaptures 
\begin{theorem}[Bernstein isoperimetric
inequality \cite{Ber}]  On the $2$-sphere $S^2\, ,$
\begin{equation} 
\begin{aligned}
L^2 & \ge 4 \pi A
(1- \frac 1{4 \pi} A)\\
&  ``="\,  \text{holds}  \\
& \text{if}\,  \text{and}\, \text{only}\, \text{if}\, 
\text{the}\,  \text{domain}\,  \text{in}\, \text{question}\,  \text{is}\,  \text{a}\,  \text{disk}. \end{aligned}\label{8.14} 
\end{equation}
\end{theorem}

\medskip

\begin{remark} For a generalization of isoperimetric inequality to  $n$-dimensional integer multiplicity rectifiable current in $\mathbb R^{n+k}$, which follows from the deformation theorem in geometric measure theory, 
we refer to Federer and Fleming (\cite {FF}).

\end{remark}

\section{Convexity and Jensen's inequalities}
 
We note by Proposition \ref{P: 8.1}, every convex function $f$ enjoys an Average Principle of Convexity and Jensen's inequality in an average sense.
From the duality between discreteness and continuity, we consider Jensen's inequality involving normalized exponential Yang-Mills energy functional $\mathcal {YM}_e ^0$.
\smallskip

Let $M$ be a compact manifold and $E$ be a vector bundle over $M\, .$ Denote  $\mathcal L_1^p(E)$ the Sobolev space of  connections of $E$ which are $p$-integrable and so are their first derivatives.
Set \begin{equation}\mathcal W (E) = \bigcap _{p \ge 1} \mathcal L_1 ^p (E) \cap \{\nabla: \mathcal {YM}_e ^0 (\nabla) < \infty\}\, . \label{9.1}\end{equation}

\begin{theorem} (Jensen's inequality involving normalized exponential Yang-Mills energy functional $\mathcal {YM}_e ^0$)  Let $\nabla$ be a connection in $\mathcal{W} (E)$. 
Then $\big ($applying \eqref{8.1} yields$\big )$
\begin{equation}
\exp \bigg ( \frac {1}{\text{Volume}(M)} \int _M \frac {||R^\nabla||^2}{2}\, dv \bigg) - 1 \le \frac {1}{\text{Volume}(M)} \int _M \big (\exp (\frac {||R^\nabla||^2}{2})- 1\big )\, dv\, .\label{9.2}\end{equation}
\begin{equation} \quad \text{That}\quad \text{is},\quad  \exp \bigg ( \frac {1}{\text{Volume}(M)} \mathcal {YM}(\nabla)\bigg) - 1 \le \frac {1}{\text{Volume}(M)} \mathcal {YM}_e ^0(\nabla)\, .\label{9.3}\end{equation}
Equality is valid if and only if $||R^\nabla||$ is constant almost everywhere.\label{T: 9.1}
\end{theorem}

\begin{proof}
This is a form of Jensen's inequality for the convex function $e^t - 1$(c.f. \cite [p.21] {Mo}). 
\end{proof}

\begin{theorem}
Let $\nabla$ be a minimizer in $\mathcal{W} (E)$ of the Yang-Mills functional $\mathcal {MY}$, and the norm $||R^\nabla||$ be constant almost everywhere. 
Then the same connection $\nabla$ is a minimizer of the normalized exponential Yang-Mills functional $\mathcal {YM}_e ^0\, ,$ and for any minimizer $\tilde {\nabla}$ of the normalized exponential Yang-Mills functional $\mathcal {YM}_e ^0\, $ in $\mathcal{W} (E)$,
the norm $||R^{\tilde {\nabla}}||$ is almost everywhere constant. \label{T: 9.2}
\end{theorem}

\begin{proof} By the definition of minimizer  $\nabla$, the monotone of $t \mapsto e^t -1\, ,$ and Jensen's inequality \eqref{9.3}, we have for each $\tilde {\nabla}$ in $\mathcal{W} (E)$, 
\begin{equation} 
\aligned
\exp \bigg ( \frac {1}{\text{Volume}(M)} \mathcal {YM}(\nabla)\bigg) - 1 & \le  \exp \bigg ( \frac {1}{\text{Volume}(M)} \mathcal {YM}(\tilde {\nabla})\bigg) - 1\\
& \le \frac {1}{\text{Volume}(M)} \mathcal {YM}_e ^0(\tilde {\nabla})\, . \endaligned \label{9.4}
\end{equation}
so that 
\begin{equation} 
\exp \bigg ( \frac {1}{\text{Volume}(M)} \mathcal {YM}(\nabla)\bigg) - 1 \le \inf _{\tilde {\nabla} \in \mathcal{W} (E)}\frac {1}{\text{Volume}(M)} {\mathcal {YM}}_e ^0(\tilde {\nabla})\, . \label{9.5}
\end{equation}
On the other hand, since $||R^{\nabla}|| =$ constant a.e.,
\begin{equation} 
\frac {1}{\text{Volume}(M)} \mathcal {YM}_e ^0({\nabla}) = \exp (\frac {||R^\nabla||^2}{2})- 1 = \exp \bigg ( \frac {1}{\text{Volume}(M)} \mathcal {YM}(\nabla)\bigg) - 1\label{9.6}
\end{equation}
so that $\nabla$ is also a minimizer of the normalized exponential Yang-Mills functional $\mathcal {YM}_e ^0\, .$

Now we assume that $\tilde {\nabla}$ is any minimizer of the normalized exponential Yang-Mills functional $\mathcal {YM}_e ^0\, $ in $\mathcal{W} (E)$. Then 
\begin{equation} 
\frac {1}{\text{Volume}(M)} {\mathcal {YM}}_e ^0(\tilde {\nabla}) \le \frac {1}{\text{Volume}(M)} \mathcal {YM}_e ^0({\nabla})\label{9.7}
\end{equation}
and combining \eqref{9.7}, \eqref{9.6} and \eqref{9.4}, 
allows us to improve  
all inequalities in \eqref{9.4}
to equalities, so that we are ready to apply Theorem \ref{T: 9.1} and conclude that  
$||R^{\tilde {\nabla}}||$ is constant almost everywhere.

\end{proof}

\section{$p$-Yang-Mills fields}

Similarly, we set $$ \mathcal W ^p (E) = \mathcal L_1 ^p (E) \cap  \mathcal L_1 ^2 (E)\, , p \ge 2$$ and obtain via \eqref{8.1}                                                                                                                                                                                                                                                                                                                                                                                                                  \smallskip

\begin{theorem}[Jensen's inequality involving $p$-Yang-Mills energy functional $\mathcal {YM}_p,\ p \ge 2$]  Let $\nabla$ be a connection in $ \mathcal W ^p (E)$.
Then
\begin{equation}
\frac 1p \bigg ( \frac {2}{\text{Volume}(M)} \int _M \frac {||R^\nabla||^2}{2}\, dv \bigg)^{\frac p2} \le \frac {1}{\text{Volume}(M)} \int _M  (\frac {||R^\nabla||^p}{p})\, dv\, .\label{10.1}\end{equation}
\begin{equation} \quad \text{That}\quad \text{is},\quad  \frac 1p \bigg ( \frac {2}{\text{Volume}(M)} \mathcal {YM}(\nabla) \bigg )^{\frac p2}\le \frac {1}{\text{Volume}(M)} \mathcal {YM}_p(\nabla)\, .\label{10.2}\end{equation}
Equality is valid if and only if $||R^\nabla||$ is constant almost everywhere.\label{T: 10.1}
\end{theorem}

\begin{proof}
This is a form of Jensen's inequality for the convex function $t \mapsto \frac 1p (2t)^{\frac p2}\, , p \ge 2$ (c.f. \cite [p.21] {Mo}). 
\end{proof}

\begin{theorem}
Let $\nabla$ be a minimizer in $\mathcal W ^p (E)$ of the Yang-Mills functional $\mathcal {MY}$, and the norm $||R^\nabla||$ be constant almost everywhere. 
Then the same connection $\nabla$ is a minimizer of the $p$-Yang-Mills functional $\mathcal {YM}_p\, ,$ and for any minimizer $\tilde {\nabla}$ of the  $p$-Yang-Mills functional $\mathcal {YM}_p\, $ in $\mathcal{W}^p (E)$,
the norm $||R^{\tilde {\nabla}}||$ is almost everywhere constant. \label{T: 10.2}
\end{theorem}

\begin{proof} By the definition of minimizer  $\nabla$, and Jensen's inequality \eqref{10.2}, we have for each $\tilde {\nabla}$ in $\mathcal{W}^p (E)$, 
\begin{equation} 
\aligned
\frac 1p \bigg ( \frac {2}{\text{Volume}(M)} \mathcal {YM}(\nabla) \bigg )^{\frac p2}  & \le  \frac 1p \bigg ( \frac {2}{\text{Volume}(M)} \mathcal {YM}(\tilde {\nabla}) \bigg )^{\frac p2} \\
& \le \frac {1}{\text{Volume}(M)} \mathcal {YM}_p(\tilde {\nabla}). \endaligned \label{10.3}
\end{equation}
so that 
\begin{equation} 
\frac 1p \bigg ( \frac {2}{\text{Volume}(M)} \mathcal {YM}(\nabla) \bigg )^{\frac p2} \le \inf _{\tilde {\nabla} \in \mathcal{W}^p (E)}\frac {1}{\text{Volume}(M)} \mathcal {YM}_p(\tilde {\nabla}). \label{10.4}
\end{equation}
On the other hand, since $||R^{\nabla}|| =$ constant a.e.,
\begin{equation} 
\frac {1}{\text{Volume}(M)} \mathcal {YM}_p({\nabla}) = \frac {||R^\nabla||^p}{p} = \frac 1p \bigg ( \frac {2}{\text{Volume}(M)} \mathcal {YM}(\nabla) \bigg )^{\frac p2} \label{10.5}
\end{equation}
so that $\nabla$ is also a minimizer of the $p$-Yang-Mills functional $\mathcal {YM}_p\, .$

Now we assume  $\tilde {\nabla}$ is any minimizer of the $p$-Yang-Mills functional $\mathcal {YM}_p\, $ in $\mathcal{W} ^p (E)$. Then 
\begin{equation} 
\frac {1}{\text{Volume}(M)} {\mathcal {YM}}_p(\tilde {\nabla}) \le \frac {1}{\text{Volume}(M)} \mathcal {YM}_p({\nabla})\label{10.6}
\end{equation}
and combining \eqref{10.6}, \eqref{10.5} and \eqref{10.3} allows us to improve  
all inequalities in \eqref{10.3}
to equalities, so that we are ready to apply Theorem {9.1} and conclude that  
$||R^{\tilde {\nabla}}||$ is constant almost everywhere.

\end{proof}

\begin{remark}J. Eells and L. Lemaire first derive Jensen's inequality and establish its optimality in the setting of exponentially harmonic maps (\cite {EL}).
F. Matsuura and  H. Urakawa show $$\exp \bigg ( \frac {\mathcal {YM}(\nabla)}{\text{Volume}(M)}\bigg)  \le \frac {\mathcal {YM}_e(\nabla)}{\text{Volume}(M)}\ \ \text{for any}\ \  \nabla \in \mathcal{W} (E)\, ,$$ and the validity of equality (\cite {MU}). 
\end{remark}

\section{An extrinsic average variation method and $\Phi_{(3)}$-harmonic maps}

We propose an extrinsic, average variational method as an approach to confront and resolve problems in global, nonlinear analysis and geometry $($cf. \cite {W1, W3}$)$. 
In contrast to an average method in PDE that we applied in \cite {CW3} to obtain sharp growth estimates for warping functions in multiply warped product manifolds, we employ \emph {an extrinsic average variational method} in
the calculus of variations $($\cite {W3}$)$, 
find a large class of manifolds of positive Ricci curvature that enjoy rich properties, and introduce the notions of \emph {superstrongly unstable $(\operatorname{SSU})$ manifolds} and \emph {$p$-superstrongly unstable $(p$-$\operatorname{SSU})$ manifolds} $($\cite {W5,W2,W4,WY}$)$.
\begin{definition}\label{D: 11.1}
A Riemannian manifold $M$ with its Riemannian metric $\langle \, , \, \rangle _M$ is said to be
{\bf superstrongly unstable (SSU)} , if there exists an
isometric immersion of $M$ in $(\mathbb R^q, \langle \, \cdot\, \rangle _{\mathbb R^q})$ with its second fundamental form $B$, such that for every {\sl unit} tangent vector $v$
to $M$ at every point $x\in M$, the following symmetric linear operator $Q^M_x$
is {\sl negative definite}.
\begin{equation}\langle Q^M_x(v),v\rangle_M=\sum^m_{i=1} \bigg (2
\langle B(v,e_i), B(v, e_i)\rangle _{\mathbb R^q}  -
\langle B(v,v),  B(e_i, e_i)\rangle _{\mathbb R^q} \bigg )\label{11.1}\end{equation}
and $M$ is said to be
$p${\bf -superstrongly unstable ($p$-SSU)} for $p\geq 2$ if the following
functional is {\sl negative valued}.
\begin{equation}\label{11.2}
F_{p,x}(v)=(p-2)\langle \mathsf B(v,v), B(v,v)\rangle _{\mathbb R^q} + \langle Q^M_x(v),v\rangle _M,
\end{equation}
where $\lbrace e_1, \ldots, e_m\rbrace$ is a local orthonormal frame on $M$.
\end{definition}

We prove, in particular that every compact $\operatorname{SSU}$ manifold must be strongly unstable $(\operatorname{SU})$, i.e., $(\rm a)$ A compact $\operatorname{SSU}$ manifold cannot be the target of any nonconstant stable harmonic maps from any manifold, $(\rm b)$ The homotopic class of any map from any manifold into a compact $\operatorname{SSU}$ manifold contains elements of arbitrarily small energy $E$,
$(\rm c)$  A compact $\operatorname{SSU}$ manifold cannot be the domain of any nonconstant stable harmonic map into any manifold, and $(\rm d)$ The homotopic class of any map from a compact $\operatorname{SSU}$ manifold into any manifold contains elements of arbitrarily small energy $E$ $($cf. \cite [Theorem 2.2, p.321] {HoW2}$)$. \smallskip

\subsection{Harmonic maps and $p$-harmonic maps, from a viewpoint of the first  elementary symmetric function $\sigma _1$}\label{S: 11.1}$\qquad$
\smallskip

\noindent
 We recall at any fixed point $x_0 \in M\, ,$ a symmetric $2$-covariant tensor field $\alpha$ on $(M,g)$ in general, or the pullback metric  $u^{\ast}$ in particular, has the eigenvalues $\lambda$ relative to the metric $g$ of $M$; i.e., the $m$ real roots of the equation $$\det (g_{ij} \lambda - \alpha_{ij}) = 0\ \operatorname{where}\ g_{ij} = g(e_i,e_j),\  \alpha_{ij} = \alpha(e_i,e_j)\, ,$$ 
 and $\{e_1, \cdots e_m\}$ is a basis for $T_{x_0}(M)\, $ $($cf.,e.g., \cite {HW}$)$.
 
 \noindent
A harmonic map $u: (M,g) \to (N,h)$ can be viewed as a critical point of the energy functional, given by the integral of a half of first elementary symmetric function $\sigma _1\, ,$ of engenvalues relative to the metric $g$, or the trace of the pulback metric tensor $u^{\ast} h$, with respect to $g$, where $\{e_1, \cdots , e_m\}$ is an local orthonormal frame field on $M$.  That is,

\begin{equation}
E(u)=\int_ M \frac 12 \sum_{i=1}^m h\big (du (e_i), du(e_i)\big ) \, dv
  = \int _M {\frac 12} {\big (\sigma _1(u^{\ast})\big )}\, dv.
\label{11.3}
  \end{equation}
  \smallskip

\noindent
A $p$-harmonic map can be viewed as a critical point of the $p$-energy functional $E_p(u)$, given by the integral of $\frac 1p$ times  $\sigma _1 $ or the trace of the pullback metric tensor to the power ${\frac p2}$, i.e.,
\begin{equation}
E(u) = \int_ M \frac 1p \bigg (\sum_{i=1}^m h\big (du (e_i), du(e_i)\big ) \bigg )^{\frac p2} \, dv= \int _M\,  \frac {1}{p} \big ({\sigma _1} (u^{\ast})\big )^{\frac p2}\, dv.
\label{11.4}
\end{equation}
\smallskip

\noindent
For the study of the stability of harmonic maps $($ resp. $p$-harmonic maps $)$, Howard and Wei $($ \cite {HoW2} $)$ $\big ($ resp. Wei and Yau $($\cite {WY}$)$ $\big )$ introduce the following notions: 

\begin{definition}
A Riemannian manifold $M$ is said to be {\it strongly unstable} $(\operatorname{SU}) \big ($resp. {\it $p$-strongly unstable} $(p$-$\operatorname{SU}) \big )$ if $M$ is neither the domain nor the target of any nonconstant smooth stable harmonic map, (resp. stable $p$-harmonic map), and the homotopic class of maps
 from or into $M$ contains a map of arbitrarily small energy $E$ (resp. $p$-energy $E_p$).
\end{definition}

This definition leads to 

\begin{theorem} Every compact superstrongly unstable $(\operatorname{SSU})$-manifold $\big ($ resp. $p$-superstrongly unstable $(p$-$\operatorname{SSU})\big )$ manifold is strongly unstable $(\operatorname{SU})\, .\big ($ resp. $p$-strongly unstable $(p$-$\operatorname{SU})\big )\, .$\label{T: 11.3}
\end{theorem}

And, we make the following classification. 

\begin{theorem}[\cite {O, HoW}]\label{T: 11.4} Let $M$ be a  compact
irreducible symmetric space.  The following statements are
equivalent:

\begin{enumerate}
\item  $M$ is $\operatorname{SSU}$.

\item $M$ is $\operatorname{SU}$.

\item  $M$ is $\operatorname{U}$; i.e. $\operatorname{Id}_M$ is an unstable harmonic map.

\item   $M$ is one of the following:
\end{enumerate}

\begin{equation}
\begin{aligned}
& {\rm(i)}\  \text{the simply connected simple Lie groups}\quad
(A_l)_{l\geq 1},\quad B_2=C_2\quad \operatorname {and}\quad (C_l)_{l\geq 3};\\
& {\rm(ii)}\  SU(2n)/Sp(n),\quad n\geq 3;\\
& {\rm(iii)}\ \text{Spheres}\quad S^k,\quad k>2;\\
& {\rm(iv)}\  \text{Quaternionic Grassmannians}\quad
Sp(m+n)/Sp(m)\times Sp(n), m\geq n\geq 1;\\
& {\rm(v)}\  E_6/F_4;\\
& {\rm(vi)}\  \text{Cayley Plane}\quad  F_4/Spin(9)\,.\end{aligned}\label{11.5}
\end{equation}
\end{theorem}

\begin{theorem}[Topological Vanishing Theorem]\label{T: 11.5}
Suppose that $M$ is a compact $\operatorname{SSU} ( \operatorname{resp.}\,  p$-$\operatorname{SSU}\, )$ manifold. Then $M$ is $\operatorname{SU}$ and

\begin{equation} \begin{aligned}
\pi_1 (M) & = \pi_2 (M) = 0\\
\big (\operatorname{resp.}\, \pi_1(M) & = \cdots = \pi _{[p]} = 0\, \big ).
\end{aligned} \label{11.6}
\end{equation}   
Furthermore,  the following three statements are equivalent:
\begin{equation}
\begin{aligned}
& {\rm(a)}\ \pi_1 (M)  = \pi_2 (M) = 0\, .\\
& {\rm(b)}\ \text{the infimum of the energy $E$ is $0$ among maps homotopic to the identity on}\, M\, . \\
& {\rm(c)}\  \text{the infimum of the energy $E$  is $0$ among maps homotopic to a map from}\,  M\, .
\end{aligned} \label{11.7}
\end{equation}
That is,

\begin{equation} \begin{aligned}
\pi_1 (M)  = \pi_2 (M) = 0 & \overset {[\operatorname{Wh}]} \Longleftrightarrow  \inf \{ E(u^{\prime}) : u^{\prime} \text{is homotopic to}\,  \operatorname{Id}\, \text{on}\, M\},\\ 
                                           &   \overset {[\operatorname {EL2}]}  \Longleftrightarrow   \inf \{ E(u^{\prime}) : u^{\prime} \text{is homotopic to}\,  u : M \to \bullet\}
\end{aligned}  \label{11.8}
\end{equation}
\end{theorem}
$($Cf. \cite [the diagram on p.58]{W3}.$)$ 

\subsection{$\Phi$-harmonic maps, from a viewpoint of the second elementary symmetric function $\sigma _2$\, (\cite {HW})}$\qquad$
\smallskip

\noindent
We introduce the notion of  $\Phi$-{\it harmonic map} which is the second symmetric function $\sigma_2$ of the pullback metric tensor $u^{\ast} h$, an analogue of $\sigma_1$ in the above subsection \ref{S: 11.1}. 

In \cite {HW}, Han and Wei show that the extrinsic average variational method in
the calculus of variations employed in the study of harmonic maps, $p$-harmonic maps, $F$-harmonic maps and Yang-Mills fields can be extended to the study of $\Phi$-harmonic maps. In fact, we find a large class of manifolds with rich properties, \emph {$\Phi$-superstrongly unstable $(\Phi$-$\operatorname{SSU})$ manifolds}, establish their links to $p$-$\operatorname{SSU}$ manifolds and topology, and apply the theory of $p$-harmonic maps, minimal varieties and Yang-Mills fields to study such manifolds. With the same notations as above, we introduce the following notions:

\begin{definition}
A Riemannian manifold $(M^m,g)$ with a Riemannian metric $g$ is said to be $\Phi$-superstrongly unstable $(\Phi$-$\operatorname{SSU})$ if there exists an isometric immersion $\mathbb R^q$ such that, for all unit tangent vectors $v$ to at every point $x\in M^m$, the following functional is always negative:
\begin{equation}
F_{{\Phi}_x}(v)=\sum_{i=1}^m \big (4\langle B(v,e_i),B(v,e_i)\rangle-\langle B(v,v),B(e_i,e_i)\rangle\big ), 
\label{11.9}
\end{equation}
where $B$ is the second fundamental form of $M^m$ in $\mathbb R^q$, and $\{e_1,\cdots,e_m\}$ is a local orthonormal frame on $M$ near $x$.
\end{definition}

\begin{definition}
A Riemannian manifold $M$ is $\Phi$-strongly unstable $(\Phi$-$\operatorname{SU})$ if it is neither the domain nor the target of any nonconstant smooth $\Phi$-stable stationary map, and the homotopic class of maps from or into $M$ contains a map of arbitrarily small energy.
\end{definition}

\begin{theorem} Every compact $\Phi$-superstrongly unstable $(\Phi$-$\operatorname{SSU})$ manifold  is $\Phi$-strongly unstable $(\Phi$-$\operatorname{SU})\, .$\label{T: 11.8}
\end{theorem}

\subsection{$\Phi_{S}$-harmonic maps, from a viewpoint of an extended second symmetric function $\sigma _2$\, (\cite {FHLW})}$\qquad$
\smallskip

\noindent
We introduce the notion of $\Phi_{S}$-harmonic maps, which is a
$\sigma _2$ version of the stress energy tensor $S$.$\qquad$
\smallskip

\noindent
In \cite {FHLW}, Feng, Han, Li,  and Wei show that the extrinsic average variational method in
the calculus of variations employed in the study of $\sigma_ 1$ and $\sigma _2$ versions of the pullback metric $u^{\ast} h$ on $M$ can be extended to the study of a $\sigma _2$ version of the stress  energy tensor $S$.
In fact, we find a large class of manifolds, {\it $\Phi_S$-superstrongly unstable $(\Phi_S$-$\operatorname{SSU})$ manifolds}, introduce the notions of a stable $\Phi_S$-harmonic map, $\Phi_S$-strongly unstable $($$\Phi_S$-$\operatorname{SU}$$)$ manifolds,
and prove 

\begin{theorem} Every compact $\Phi_S$-superstrongly unstable $(\Phi_S$-$\operatorname{SSU})$ manifold  is $\Phi_S$-strongly unstable $(\Phi_S$-$\operatorname{SU})\, .$\label{T: 11.9}
\end{theorem}

\subsection{$\Phi_{S,p}$-harmonic maps, from a viewpoint of a combined extended second symmetric function $\sigma _2$\, (\cite {FHW})}$\qquad$
\smallskip

\noindent
We introduce the notion of $\Phi_{S,p}$-harmonic maps, which is a combined generalized
$\sigma _2$ version of the stress energy tensor $S$, and a $\sigma _1$ version of the pullback $u^{\ast}$.\smallskip

In \cite {FHLW}, Feng, Han, Li,  and Wei show that the extrinsic average variational method in
the calculus of variations employed in the study of $\sigma_ 1$ and $\sigma _2$ versions of the pullback metric $u^{\ast} h$ on $M$ and stress-energy tensor can be extended to the study of a combined extended second symmetric function $\sigma _2$ version.
In fact, we find a large class of manifolds, $\Phi_{S,p}$-superstrongly unstable $(\Phi_{S,p}$-$\operatorname{SSU})$ manifolds, introduce the notions of a stable $\Phi_{S,p}$-harmonic map, $\Phi_{S,p}$-strongly unstable $($$\Phi_{S,p}$-$\operatorname{SU}$$)$ manifolds, and prove 

\begin{theorem} Every compact $\Phi_{S,p}$-superstrongly unstable $(\Phi_{S,p}$-$\operatorname{SSU})$ manifold  is $\Phi_{S,p}$-strongly unstable $(\Phi_{S,p}$-$\operatorname{SU})\, .$\label{T: 11.10}
\end{theorem}

\subsection{$\Phi_{(3)}$-harmonic maps, from a viewpoint of the third elementary symmetric function $\sigma _3$ (\cite {FHJW})}$\qquad$
\smallskip

\noindent
 We introduce the notion of of $\Phi_{(3)}$-harmonic maps, which is a $\sigma _3$ version
of the pullback $u^{\ast}$.

In fact, Feng, Han, Jiang,  and Wei show that the extrinsic average variational method in
the calculus of variations employed in the study of $\sigma_ 1$ and $\sigma _2$ versions of the pullback metric $u^{\ast} h$ on $M$ can be extended to the study of the third symmetric function $\sigma _3$ version.
Whereas we can {view harmonic maps as $\Phi_{(1)}$-harmonic maps} (involving $\sigma_1$) and $\Phi$-harmonic maps as $\Phi_{(2)}$-harmonic maps (involving $\sigma_2$) , we introduce the notion of  a $\Phi_{(3)}$-harmonic map and find a large class of manifolds, $\Phi_{(3)}$-superstrongly unstable ($\Phi_{(3)}$-$\operatorname{SSU}$) manifolds, introduce the notions of a stable $\Phi_{(3)}$-harmonic map, $\Phi_{(3)}$-strongly unstable ($\Phi_{(3)}$-$\operatorname{SU}$) manifolds, and prove 

\begin{theorem}[\cite {FHJW}] Every compact $\Phi_{(3)}$-superstrongly unstable $(\Phi_{(3)}$-$\operatorname{SSU})$ manifold  is $\Phi_{(3)}$-strongly unstable $(\Phi_{(3)}$-$\operatorname{SU})\, .$\label{T: 11.11}
\end{theorem}

\begin{definition}[\cite {FHJW}]\label{D: 11.11}
A Riemannian manifold $M^{m}$ is said to be $\Phi_{(3)}$-superstrongly
unstable $($$\Phi_{(3)}$-$\operatorname{SSU}$$)$ if there exists an isometric immersion
of $M^{m}$ in $\mathbb{R}^{q}$ with its second fundamental form  $B$ such that for all unit
tangent vectors $v$ to $M^{m}$ at every point $x\in M^{m},$ the
following functional is negative valued.
	
\begin{equation} 
{F_{\Phi_{(3)}}}_x (v)=\sum_{i=1}^{m}\big (6\langle B(v,e_{i}),B(v,e_{i})\rangle _{\mathbb R^q}-\langle B(v,v),B(e_{i},e_{i})\rangle_{\mathbb R^q}\big ),\label{11.8}
\end{equation}
where $\{e_{1},\cdots,e_{m}\}$ is a local orthonormal
frame field on $M^{m}$ near $x$.
\end{definition}

\begin{theorem}\label{T: 11.13}
Every $\Phi_{(3)}$-$\operatorname{SSU}$ manifold $M$ is $p$-$\operatorname{SSU}$ for any $2 \le p \le 6$.
\end{theorem}
\begin{proof}
By Definition \ref{D: 11.11}, $\Phi_{(3)}$-$\operatorname{SSU}$ manifold enjoys 
\begin{equation}
\begin{aligned}
{F_{\Phi_{(3)}}}_x (v)=\sum_{i=1}^{m}\big (6\langle B(v,e_{i}),B(v,e_{i})\rangle _{\mathbb R^q}-\langle B(v,v),B(e_{i},e_{i})\rangle_{\mathbb R^q}\big )<0
\end{aligned}\label{11.7}
\end{equation}
for all unit tanget vector $v \in T_x(M)$. It follows from \eqref{11.2} and \eqref{11.8} that
\begin{equation}
\begin{aligned}
F_{p,x}(v)&=(p-2)\langle \mathsf B(v,v), B(v,v)\rangle _{\mathbb R^q} + \langle Q^M_x(v),v\rangle _M\\
& \le (p-2)\sum^n_{i=1} \bigg (2
\langle B(v,e_i), B(v, e_i)\rangle _{\mathbb R^q} \bigg )  \\
& \qquad + \sum^n_{i=1} \bigg (2
\langle B(v,e_i), B(v, e_i)\rangle _{\mathbb R^q}  -
\langle B(v,v),  B(e_i, e_i)\rangle _{\mathbb R^q} \bigg )\\
&\leq\sum_{i=1}^n\big (p\langle  B(v, e_i), B(v, e_i)\rangle-\langle B(v,v), B( e_i, e_i)\rangle\big ) \\
&\leq\sum_{i=1}^n\big (6\langle B(v,e_i), B(v, e_i)\rangle-\langle  B(v,v),B( e_i, e_i)\rangle\big )<0,
\end{aligned}
\end{equation}
for $2\leq p\leq 6$.
So by Definition \ref{D: 11.1}, $M$ is $p$-SSU for any $2\leq p\leq 6$.\\
\end{proof}

\begin{theorem}\label{T: 11.14}
Every compact $\Phi_{(3)}$-$\operatorname{SSU}$ manifold $M$ is $6$-connected\, , i.e., 

\begin{equation}
\pi_1(M) = \cdots = \pi_{6}(M) = 0.
\end{equation}

\end{theorem}

\begin{proof}

Since every compact $p$-SSU is $[p]$-connected (cf.  \cite [Theorem 3.10 , p. 645] {W5}), and $p=6$ by the previous Theorem, the result follows. 
\end{proof}

\begin{theorem}[Sphere Theorem]\label{T: 11.15}
Every compact $\Phi_{(3)}$-$\operatorname{SSU}$ manifold $M$ of dimension $m \le 13$ is homeomorphic to an $m$-sphere.
\end{theorem}
\begin{proof}
In view of Theorem \ref{T: 11.13}, $
M$ is 6-connected. By the Hurewicz isomorphism theorem, the 6-connectedness of $M$ implies homology groups $H_1(M)=\cdots=H_6(M)=0$. It follows from Proincare Duality Theorem and the Hurewicz Isomorphism Theorem (\cite{SP}) again, $H_{m-6}(M)=\cdots=H_{m-1}(M)=0$, $H_{m}(M)\neq 0$ and $M$ is ($m-1$)-connected. Hence $N$ is a homotopy $m$-sphere. Since $M$ is $\Phi_{(2)}$-SSU, $m\geq 7$. Consequently, a homotopy $m$-sphere $M$ for $m \ge 5$ is homeomorphic to an $m$-sphere 
by a Theorem of Smale (\cite{Sm}).
\end{proof}

We summarize some of new manifolds found and these results obtained by an extrinsic average method in Table 1 in Section 1.

\bibliographystyle{amsalpha}

\end{document}